\documentclass{amsart}
 \usepackage{url}
\usepackage[colorlinks,citecolor=blue,urlcolor=blue, linkcolor=blue]{hyperref}

\usepackage[title]{appendix}
\usepackage{bbm}
\usepackage{amscd}
\usepackage[all]{xy}
 \usepackage{amsmath}
\usepackage{amssymb}
\usepackage{amsthm}
\usepackage[sort&compress,numbers]{natbib}
\usepackage{tikz}
\usepackage{tikz-cd}
\usepackage{wrapfig}
\usepackage{xcolor}
\usepackage{framed}
\newtheorem{theorem}{Theorem}[section]
\newtheorem{proposition}[theorem]{Proposition}

\newtheorem{lemma}[theorem]{Lemma}
\newtheorem{corollary}[theorem]{Corollary}
\theoremstyle{definition}
\newtheorem{claim}{Claim}
\newtheorem{definition}[theorem]{Definition}
\newtheorem{cnj}[theorem]{Conjecture}

\newtheorem{remark}[theorem]{Remark}
\newtheorem{problem}[theorem]{Problem}

\newtheorem{fact}[theorem]{Fact}

\AtBeginDocument{   \def\MR#1{}}

\newcommand{\NN}{{\mathbb{N}}}

\newcommand{\sub}{\subseteq}

\newcommand{\bi}{\begin{itemize}}
\newcommand{\ei}{\end{itemize}}
\newcommand{\bc}{\begin{center}}
\newcommand{\ec}{\end{center}}

\newcommand{\ES}{\emptyset}

\newcommand{\la}{\langle}
\newcommand{\ra}{\rangle}

\newcommand{\n}{\noindent}

\newcommand{\lland}{\, \land \, }

\newcommand\+[1]{\mathcal{#1}}

\newcommand{\LR}{\Leftrightarrow}
\newcommand{\RA}{\Rightarrow}
\newcommand{\LA}{\Leftarrow}

\newcommand{\rapf}{\n $\RA:$\ }
\newcommand{\lapf}{\n $\LA:$\ }

\renewcommand \hat \widehat

\begin{document}

\title{Computably locally compact groups and their closed subgroups}

\author{Alexander G.~Melnikov}
\email{sasha.melnikov@vuw.ac.nz}
\author{Andr{\'e} Nies}
\email{andre@cs.auckland.ac.nz}

\maketitle

\begin{abstract}
 Given a computably locally compact Polish space $M$,  we show that  its  1-point compactification  $M^*$   is computably compact. Then, for a computably locally compact group $G$, we show that  the  Chabauty space $\mathcal S(G)$  of closed subgroups of $G$ has a canonical effectively-closed (i.e., $\Pi^0_1$) presentation as a subspace of the  hyperspace  $\mathcal K(G^*)$  of closed sets of  $G^*$. We construct a computable discrete abelian group $H$ such that $\mathcal S(H)$ is  not computably closed in $\mathcal K(H^*)$; in fact, the only computable points of  $\mathcal S(H)$ are   the trivial group and $H$ itself,  while  
$\mathcal S(H)$ is uncountable.
In the case that  a computably locally compact group $G$ is also totally disconnected,   we provide  a further  algorithmic characterization of $\mathcal S(G)$  in terms of the countable meet groupoid of $G$ introduced recently by the authors  (arXiv: 2204.09878). We apply our results and techniques to  show that  the index set of the computable   locally compact abelian groups that contain  a closed subgroup isomorphic to $(\mathbb{R},+)$ is arithmetical. \end{abstract}

\setcounter{tocdepth}{1}
\tableofcontents

\section{Introduction}
Computable mathematics probes the algorithmic content of mathematical concepts and results. 
Our  paper contributes to  the study of computable Polish  groups initiated in \cite{MeMo, pontr}, with special attention to   locally compact groups.
Computable profinite groups are first studied in  Metakides and Nerode
\cite{MetNer79}, La Roche \cite{LaRo1,LaRothesis} and
Smith \cite{Smith1,SmithThesis}.    For   recent results on  computable topological groups  in various classes we cite~\cite{MeMo,sinf,lupini,  CoToGr} and \cite{ArnoHaar}.  



 
Let   $G$ be a locally compact (l.c.) Polish group.  The Chabauty space $\mathcal S(G)$  is a compact topological space on the collection of all closed subgroups of $G$; its basic open sets are the collections of closed  subgroups that are disjoint from a compact set and meet finitely many open sets.  It   plays a significant role in the theory of locally compact groups; see, e.g.,~\cite{Cornulier:11}.


 In the present article we investigate the computability-theoretic aspects of the Chabauty space $\mathcal S(G)$ of a computable  locally compact (Polish) group $G$. We establish that   $\mathcal S(G)$ can be uniformly represented as an effectively closed subset of a computably compact space, but this subset may fail to be computable.  We then apply our computability-theoretic  analysis of the Chabauty space to derive a  classification-type result about (not necessarily computable) groups.

We proceed to  an informal overview and the statements of our three   results. Most notions will only be defined formally in subsequent sections. 
For instance,   the standard notions of a computably (locally) compact Polish spaces, as well as their   computably closed  and effectively closed  ($\Pi^0_1$) subsets  will be given in the Section~\ref{s: prelim}. 
All our groups and spaces will be Polish.

\subsection{The Chabauty of $G$ as a subspace  of  $\mathcal K(G^*)$}   
 

Let $G^*$  be  the   1-point compactification  of the space underlying the group $G$, and let $\mathcal K(G^*)$ be  the hyperspace of closed subsets of $G^*$. The Chabauty space  $\mathcal S(G)$   can be naturally seen as a closed subset of $\mathcal K(G^*)$; for detail see Subsection~\ref{subs:folklorestuff}. 

In Section~\ref{subs: 1-point compactification} we give a detailed   treatment of the effectiveness of how to take the 1-point compactification of a computably locally compact (l.c.) metric space.
The material of   Section~\ref{subs: 1-point compactification} is not particularly elementary, and it appears to be new in the form in which we need it. 
We believe that these results, including even some technical results of the section (such as Lemma~\ref{lemma:sigma} and Proposition~\ref{cor:HB}) are  important for the general framework of computable metric space theory.
For instance, in Theorem~\ref{thm:1po} we show that, for a computably l.c.\ space~$M$, its 1-point compactification $M^*$ is a computably compact metric space.
It follows that
the hyperspace  $\mathcal K(M^*)$ of its closed subsets is computably compact; this is  Fact~\ref{fact:ccc}. 

For a computably locally compact group~$G$, the question arises 
to what extent is  $\mathcal S(G)$ effective as  a subset $\mathcal K(G^*)$.
Our  first   result answers this as follows.

\begin{theorem}\rm\label{main:thm}
Let  $G$ be a computably locally compact group.
\begin{enumerate}
\item  $\mathcal S(G)$ is  $\Pi^0_1$ in the computably compact space $\mathcal K(G^*)$.
\item  $\mathcal S(G)$ is not in general  computably closed in $\mathcal K(G^*)$.
\end{enumerate}
\end{theorem}
\noindent 
Our construction of the $\Pi^0_1$ presentation of $\mathcal S(G)$ has a number of further algorithmic features; the construction will appear in the proof of Proposition~\ref{thm:cha}.
One such   feature is Lemma~\ref{lem:points}:  for a closed subgroup $H$ of  $G$,
 $H \leqq_c G$ is computably closed if, and only if, 
$H \cup \{\infty\}$ is a computable point in $\mathcal K(G^*)$, with all possible uniformity.

We note that  (2) of Theorem~\ref{main:thm} is witnessed by a discrete torsion-free abelian group  that has only the two  trivial computable subgroups, yet  has uncountably many 
  subgroups (Theorem~\ref{thm:notsimple}). This   makes $\mathcal S(G)$ uncountable too.
In particular, it follows from  Lemma~\ref{lem:points} that $\mathcal S(G)$ has only two  computable  points; they correspond to the trivial subgroups of $G$.

\subsection{The t.d.l.c.~case}
Our second   result is concerned with   totally disconnected locally compact (t.d.l.c.) Polish groups. 
This class, extensively studied in the literature,  is  the narrowest class  of Polish groups  that contains both the countable discrete and the profinite  Polish groups;  recent papers include   \cite{tdlc1,tdlc2,tdlc3,tdlc4,tdlc5}.
In \cite{tdlc}, the authors have initiated a systematic study of computably t.d.l.c.~groups; see also \cite{lupini,separ}.
As we explain in Section~\ref{sec:mimi} in detail,
 there is a canonical way to pass from a computably locally compact presentation of a t.d.l.c.~group
to an effective presentation of the  \emph{meet groupoid} $\mathcal{W}(G)$ of all compact open cosets in $G$, a countable structure (see Theorem~\ref{thm:ma}). Since $G$ is t.d.l.c., it has  a neighbourhood basis of the identity consisting of compact open subgroups; this leads to 
a computable duality   between t.d.l.c.\ groups  $G$ and meet groupoids  $\mathcal{W}(G)$~\cite{tdlc}. 

We are interested in a description of the compact space $\mathcal S(G)$ based on the purely algebraic structure $\mathcal W(G)$. The  algebraic  notion of a closed-subgroup ideal  of a meet groupoid will be defined below (Def.\ \ref{df:ideals}).
\begin{theorem}\label{thm:gr}
Suppose $G$ is a computably locally compact t.d.l.c.~group.  For a closed subgroup $H$ of $G$,  the following are equivalent:


\begin{enumerate}
\item $H$ is computably closed;


\item  $H$ corresponds to a computable closed-subgroup ideal  in the (computable) dual meet groupoid $\mathcal{W}(G)$ of $G$.

\end{enumerate}
This correspondence is computably uniform\footnote{We usually think if  $G$ as given by a  computable locally compact presentation, and  of  $\mathcal S(G)$ as given by the presentation  produced in Theorem~\ref{main:thm}. We also think of  the   meet groupoid $\mathcal{W}(G)$ as given by an    effective presentation  that is produced combining the aforementioned results from \cite{tdlc,separ}; the technical details will be explained   in Section~\ref{sec:mimi}.}.
\end{theorem}

 The correspondence in (2) above is (in a certain sense) canonical; this will be clarified in Definition~\ref{def:id}
and   the subsequent discussion. 
As an immediate consequence of Theorem~\ref{main:thm}(2), we obtain that there is a computably t.d.l.c.~$G$ so that only the trivial closed-subgroup ideals of $\mathcal{W}(G)$ are computable.  

We  believe that Theorem~\ref{thm:gr} and Lemma~\ref{lem:points} mentioned earlier
illustrate that the effective presentation of $\mathcal S(G)$ constructed in the present paper is   (algorithmically) natural and robust. 
Our third result applies this presentation to measure the complexity of the recognition problem for certain l.c.~groups; we discuss this next.

\subsection{An application to index sets}

The approach to classification problems via index sets is standard in computable structure theory in the countable case; for a comprehensive introduction to this approach, see~\cite{GonKni}. 
 Early applications of index sets in analysis can be found in~\cite{CRindex}, and for many more recent results see \cite{surfaces,norm2}  
 and  the recent survey~\cite{MDsurvey}.
 Our third   result   essentially states that \emph{it is arithmetical to tell whether a given  locally compact abelian group contains a closed subgroup isomorphic to $\mathbb{R}$.}
We informally explain what we mean, and then   formally state the result. 
Fix a computable enumeration $(G_i)_{i \in \omega}$ of all partial computable Polish groups. Each   $G_i$ is given by a (potential) computable pseudo-metric on $\mathbb{N}$ and Turing operators potentially representing group operations upon this space. We are interested in   
the complexity of the index set of~$P$, $$I_P = \{ i: G_i \mbox{ is a computable Polish group with property }P \}.$$
The general intuition is that `tractable' properties, such as being compact, tend to have arithmetical index sets. 
In contrast, if a given property $P$ is difficult to `test' for a given space, then its index set tends to be complex as well, namely at least analytic or co-analytic complete. Such estimates reflect that such `difficult' properties are intrinsically \emph{not local} and certainly  \emph{not first-order} even if we additionally allowed quantification over effective procedures.

We return to the discussion of our third   result. Recall that a metric is \emph{proper} (or Heine-Borel) if all closed balls are compact   for the topology given by   this metric. 
We say that $G$ is properly metrized (or Heine-Borel) if the metric on $G$ is proper. Each locally compact Polish space admits a compatible proper metric (\cite{regular_stuff}); 
  in Proposition~\ref{cor:HB} we will obtain  an algorithmic version of this result.
We write  
$H \leq_c G$ for `$H$ is (isomorphic to) a closed subgroup of $G$'.
The property `$\mathbb{R} \leq_c G_i$'  seems intrinsically \emph{not} first-order and \emph{not} local. Nonetheless, using Theorem~\ref{main:thm} and the techniques developed in its proof, we establish:

\begin{theorem}\label{thm:3}  The index set
$$\{i : G_i \mbox{ is  a properly metrized  abelian group and }    \mathbb{R} \leq_c G_i \}$$
 is   $\Pi^0_3$-complete.
\end{theorem}

 While  $\Pi^0_3$-hardness  is  easy to verify,  establishing the  upper bound of  $\Pi^0_3$ takes much more work.  This upper bound  makes   essential use of 
 the effective presentability of   Chabauty spaces given by    Theorem~\ref{main:thm}.

There are numerous potential questions in the area of computable Polish groups, a field that remains largely unexplored. For instance, in the context of this paper, we wonder whether the `Pontryagin-Chabauty duality' established in \cite{Cornulier:11} is computable.
 (Various  effective dualities have recently been established in \cite{bastone, lupini, EffedSurvey, uptohom, topsel}.)
 A few further questions will be stated in the subsequent sections.

 The proofs of our results are divided into several lemmas and propositions
 that combine methods of computable topology with tools from computable algebra.
 Although each individual fact is not too difficult to establish, the overall complexity of our proofs arises from the  combination of these facts and techniques.

 \n {\bf Acknowledgements.} The second author was partially supported by the Marsden fund of New Zealand, UOA-1931.





\section{Computable Polish spaces and groups} \label{s: prelim}

We use   standard computability-theoretic terminology and notation as in~\cite{Soa}, except that we write ``c.e." (computably enumerable)  instead of ``r.e." (recursively enumerable). In particular,  $0^{(n)}$ denotes the Turing degree of the $(n-1)$th iteration of the halting problem, and $\Sigma^0_n, \Pi^0_n, \Delta^0_n$ denote the   levels of the arithmetical hierarchy. It is well-known that the $\Delta^0_{n+1}$ sets are exactly the sets computable relative to $0^{(n)}$, and $\Sigma^0_{n+1}$ and $\Pi^0_{n+1}$ are the classes of $0^{(n)}$-computably enumerable ($0^{(n)}$-c.e.) and the complements of $0^{(n)}$-c.e.~sets, respectively.
We also use the standard terminology of the theory of computable Polish 9aoways viewed as metric spaces)~\cite{EffedSurvey,IlKi}; we give a brief overview of it in the subsection below.

\subsection{Computable Polish spaces}\label{subs:prelimPolish}  
A Polish space $M$ is \emph{computable Polish} or \emph{computably (completely) metrized} if there is a compatible, complete metric $d$ and a countable sequence of \emph{special points} $(x_i)$ dense in $M$ such that, on input $i, j, n$, we can compute a rational number $r$ such that $|r - d(x_i, x_j)| <2^{-n}$. 


A \emph{basic open ball} is an open ball having a rational radius and centred in a special point. Let $ X$ be a computable Polish space, and $(B_i)$ is the effective list of all its basic open balls, perhaps with repetition. 
We call
\[
N^x = \{i: x \in B_{i}\}
\]
\emph{the name of $x$} (in $ X$).
A (fast) Cauchy name of a point $x$ is a sequence $(x_n)$ of special points such that $d(x_n, x)< 2^{-n}$. It is easy to see that we can uniformly effectively turn an enumeration of  $N^x$  into a fast Cauchy name of $x$, and vice versa. 
We say that a point $x$ is computable if $N^x$ is c.e.; equivalently, if there is a computable sequence $(x_n)$ of special points that is a fast Cauchy name of $x$.

We can also use basic open balls to produce names of open sets, as follows.
A \emph{name} of an open set $U$ is a set  $W \subseteq \mathbb{N}$ such that $U = \bigcup_{i \in W} B_i$, where $B_i$ stands for the $i$-th basic open set  (basic open ball).
 If  an open $U$ has a c.e.~name, then we say that $U$ is \emph{effectively open} or \emph{c.e.~open}.
 A closed set $C$ is c.e.~if the set $\{i: B_i \cap C \neq \emptyset\}$ is c.e.; equivalently,   the set possesses a uniformly computable dense sequence of points.
 A closed set is \emph{computable}~if both the set and its complement are c.e.; see \cite{EffedSurvey,IlKi} for further details. 
 Note that this is the same as to say that the function $x \mapsto d(x, C) = \sup_{c \in C} d(x,c)$ is  computable in the following sense.


\begin{definition}
We say that a map (between computable Polish spaces $X$ and $Y$) is \emph{computable} if it satisfies any of the items from the lemma below.
\end{definition}

\begin{lemma}\label{le: computable map}
Let $f\colon X\to Y$ be a function between computable Polish spaces.
The following are equivalent:
\begin{enumerate}

\item There is an enumeration operator $\Phi$ that on input a name of an open set $Y$ (in $Y$), lists a name of $f^{-1}(Y)$ (in $ X$).
\item There is an enumeration operator $\Psi$, that given the name of $x \in X$, enumerates the name of $f(x)$ in $ Y$.     \label{part: Phi}
\item There exists a uniformly effective procedure that on input  a fast Cauchy name of $x \in M$  lists a fast Cauchy name of $f(x)$ (note that the Cauchy names need not be computable).
\end{enumerate}
\end{lemma}

The first condition in the lemma above says that the map is \emph{effectively continuous}.
We also say that $f\colon X\to Y$ is \emph{effectively open} if there is an enumeration operator that given a name of an open set $A$ in $X$, outputs a name of the open set $f(A)$ in $Y$.
In particular, if $f$ is a computable and is a homeomorphism, then it is  effectively open if, and only if, $f^{-1}$ is computable. In this case we say that $f$ is a \emph{computable homeomorphism}.  

We say that two computable metrizations on the same Polish space are \emph{effectively compatible} if the identity map on the space  is a computable homeomorphism when viewed as a map from the first metrization to the second metrization under consideration.


\subsection{Computable Polish groups} 
In the discrete (at most) countable case, Mal'cev \cite{Mal} and Rabin~\cite{abR} define a computable group as follows. A   group is \emph{computable} (recursive, constructive) if its domain is a computable set of natural numbers and the group operations are computable functions upon this set.
A~computable presentation (computable copy,  constructivization) of a group is a computable group isomorphic to it.
  The study of computable groups
   is a central subject in computable structure theory; see the books  \cite{AshKn,ErGon}. For the well-studied special case of abelian groups, see the surveys~\cite{Khi,melbsl}. We adopt the following

\begin{definition}\cite{MeMo,pontr}\label{def:compgrdef}
A computable Polish group is a computable Polish space together with computable group operations $\cdot$ and $^{-1}$.
\end{definition}

The following   elementary fact illustrates that the notions   agree in the discrete case: 

\begin{fact}[\cite{CoToGr}]\label{hgotGDDD112330OOFggrs} For a discrete  countable group $G$, the following are equivalent:
 
\begin{enumerate}
\item $G$ has a computable Polish presentation as in Def.~\ref{def:compgrdef};
\item $G$ has a computable presentation. 
\end{enumerate}

\end{fact}

The next fact will be used throughout,  often without an explicit reference.
\begin{fact}[\cite{CoToGr}]\label{fact:open}
In a computable Polish group, the multiplication and inversion operators are   effectively open.
\end{fact}

\subsection{Computable compactness} The following definition is equivalent to many other definitions of effective (computable) compactness  in the literature.

\begin{definition}\label{def:effcomp}
A compact computable Polish  space is  computably compact if there is a (partial) Turing functional that,  given a countable cover of the space,  outputs a finite subcover. 
\end{definition}

This is equivalent to saying that, for every $n$, we can uniformly produce at least one finite open $2^{-n}$-cover of the space by basic open balls. For several equivalent definitions of computable compactness, see \cite{EffedSurvey} and \cite{IlKi}. It is   well-known that, given a computable Polish space $C$ that is compact (but not necessarily computably compact), using $0'$ one can produce a sequence of basic open $2^{-n}$-covers of the space, thus making it computably compact relative to $0'$. The following  fact is folklore (see, e.g., \cite{EffedSurvey}): 


\begin{lemma}
Every computable image of a computably compact space is itself computably compact.
\end{lemma}
Both the supremum and the infimum of a computable function $f: X \rightarrow \mathbb{R}$ is computable, provided that $X$ is computably compact, and this is uniform.
We will also use that a computable closed subset of a computably compact space $M$ is itself computably compact   under the metric inherited from $M$.

\begin{lemma}\label{0jj}
Every  nonempty  co-c.e.~closed set $C$ in a computably compact Polish space $K$ has a $0'$-computable dense sequence of points. \end{lemma}

\begin{proof}  Using \cite[Prop.~4.1]{BBP},  fix a computable surjective function $f: 2^{\omega} \rightarrow K$. Since $f$ is effectively continuous by Lemma~\ref{le: computable map}, the pre-image of the effectively open complement of $C$ is effectively open  
in $2^{\omega}$. Thus,  $f^{-1}(C)$  is a $\Pi^0_1$ class  in $2^{\omega}$. Using $0'$ as an oracle we can decide whether a binary string is  extendable to a path in  $f^{-1}(C)$. This yields a $0'$-computable dense set in  $f^{-1}(C)$ and, thus, in $C$.
\end{proof}


\subsection{Computable local compactness for spaces}\label{subs:clcstuff}
 Recall that by a  name $N^x$ of a point $x$  one means the collection of all basic open sets (balls) of a computable Polish $M$ that contain~$x$. In the definition below, a compact $K \subseteq M$ is represented  by
the computable index of a procedure
 enumerating a sequence of finite open $2^{-n}$-covers so that each ball in the cover intersects $K$ non-trivially.
  (In particular, it makes $K$ a computable closed set and a computably compact subspace of the ambient computable Polish space $M$.)
  The index of such a procedure enumerating finite covers of $K$ is sometimes referred to as a \emph{compact name} of $K$.
  
  The usual definition of a locally compact space $M$ says that for every $x \in M$ there is an open set $B$ and a compact set $K$ such that $x \in B \subseteq K$. 
 We adopt the following notion of effective local compactness which is essentially the approach taken in \cite{Kamo}, up to a change of notation.

\begin{definition}\label{def:loc}
A computable Polish space $M$ is \emph{computably  locally compact} if there is an algorithmic procedure that, given (an enumeration of the name $N^x$ of) any point $x$,
outputs a basic open set $B$ and a computable compact set $K \subseteq M$ such that $x \in B \subseteq K$. We call the functional representing the procedure $N^x\rightarrow  K \supseteq B \ni x$ a \emph{computable locally compact structure} on $M$.

\end{definition}

\begin{remark} Other closely related  definitions can be found in \cite{lc1,lc2, lc3}. See \cite{lc1} for   a detailed discussion of the various approaches to effectivizing local compactness  in the literature and the subtle differences between them.  
\end{remark}
\noindent 
 Clearly, for a compact space,  computable \emph{local} compactness    is equivalent to   computable compactness.  Note that $x$  does not have to be a computable point, while the output of the procedure  in the definition has to be a computable index of the procedure representing $K$.
If the functional in the definition above is not computable but is $Y$-computable for some oracle $Y$ (e.g., $Y = \emptyset''$) then we say that $M$ admits a \emph{$Y$-computable locally compact structure}. We will need this relativized version of computable local compactness in the proof of our index set result.

 A useful reformulation of the notion of computable local compactness presented below involves the notion of a computable closed ball.
 If $B = B(\xi, r)$ is an open ball,   by  $B^c$ we denote the corresponding  closed ball $\{x: d(\xi,x) \leq r\}$; note that this  contains the closure $\overline{B}$  but may be strictly larger. Recall that we say that a ball  $B = B(\xi, r)$ is basic if $\xi$ is a special  point and $r$ is a rational (given as a fraction). If $\xi$ and $r$ are merely computable, we say that the ball is computable. 
 
 The following proposition will be   useful. See \cite[Proposition 12]{lc1} for a similar fact in a different framework. Note that the property established   is essentially  \cite[Definition 3]{lc3}.

\begin{proposition}\label{prop:ball}
In the notation of Definition~\ref{def:loc}, 
 we can assume that  $K$ equals  $B^c$  for  a computable open ball~$B$ containing~$x$. In other words, given $x$ we can produce a computable index of a computable point $\xi$ and of a computable real $r$ so that $B(\xi, r) \ni x$ and $B^c(\xi, r)$ is computably compact (and is represented by  its compact name). Furthermore, we can additionally assume that $B^c(\xi, r) = \overline{B(\xi,r)}$, where the latter is the closure of $B(\xi, r)$.

\end{proposition}

\begin{proof} 
 Proposition 3.30 of
\cite{EffedSurvey} establishes that we can uniformly produce a system  of  finite $2^{-n}$-covers $C_n$ of a computably compact space $K$ that consist of  computable open balls
that possess the following strong property. Each ball in $C_n$ is represented by the pair of indices of its computable centre and its computable radius.
Given any finite collection of (the parameters describing) balls $B_0, \ldots, B_k \in \bigcup_n C_n$, we can uniformly decide
whether 
$$\bigcap_{i \leq k} B_i  \neq \emptyset.$$
Furthermore (as explained in \cite[Remark 3.18]{EffedSurvey}), when we decide intersection, we can replace 
some (or all) balls in the sequence $B_0, \ldots, B_k$ with the respective closures or the  closed balls, e.g., 
$$B_0 \cap \overline{B_1} \cap B^c_2 \cap B_3 \neq \emptyset$$
iff 
$$B_0 \cap B_1 \cap B_2 \cap B_3 \neq \emptyset,$$
where $\overline{B_i}$ are  the closure of $B_i$ and $ B_i^c$ is the basic closed ball with the same parameters as $B_i$.
Additionally, we can further assume that every $B(a, r) \in C_{n+1}$ is formally included in some $B(b, q) \in C_{n}$, meaning that 
$d(a,b)+ r < q$.

\

The following further convenient property of such covers was not mentioned in \cite{EffedSurvey}.
\begin{claim}\label{claim:odin}
In the notation above, for any computable ball $B \in \bigcup_n C_n$ we have 
$\overline{B} = B^c.$
\end{claim}

\begin{proof} Suppose $\alpha \in B^c \setminus \overline{B}$.
The set  $U = K \setminus \overline{B} $ is open, and each  $\alpha \in U\cap B^c$ must be contained in $U$ together with some  $B' \in C_n$, where $n$ is sufficiently large.
But then we have $B' \cap  B^c \neq \emptyset$ while $B' \cap  \overline{B} = \emptyset$, contradicting the properties of $\bigcup_n C_n$ described before the claim.
\end{proof}

Recall that every ball in $\bigcup_n C_n$ is represented by a pair of indices for its computable centre and its radius.

\begin{claim}\label{claim:dva}
In the notation above, for every $B \in \bigcup_n C_n$ the closed ball $B^c = \overline{B}$ is a computably compact subspace of $K$, and this is uniform in the parameters describing $B$.
Indeed, this is witnessed by finite $2^{-n}$-covers $C_n' \subseteq C_n$ that can be found uniformly in $n$.
\end{claim}

\begin{proof}
We can list a dense sequence of uniformly computable points in $B$ which turns $B$ into a computable Polish space. Since intersection is decidable in  $ \bigcup_n C_n$
in the strong sense described earlier, we can uniformly effectively restrict each  $C_n$ to its finite subset $C'_n$ consisting of  $B' \in C_n$ having the property
$\overline{B}\cap B' \neq \emptyset$.
The resulting uniformly computable system $(C'_n)_{n \in \omega}$ of $2^{-n}$-covers of $B^c$ witnesses that $B^c$ is itself a computably compact subspace of $K$.
\end{proof}

We now return to the proof of the proposition. Let $x \in B \subseteq K$ be as in Definition~\ref{def:loc}, where $B = B(a,q) = \{y: d(a,y) < q\}$.
Fix a system of covers $(C_n)$ of $K$ with the strong properties described above, and (using $x$) search for the first found open ball 
$D = B(\xi, r) \in \bigcup_n C_n$ such that $$x \in D \subseteq B$$ as witnessed by  $d( \xi, x) < r$ and $d(a,\xi) +r < q$. (Note the latter implies $D^c \subseteq B$.)
This search is uniformly effective in  the name $N^x$ of $x$ and will eventually terminate.

 The closure of $D$ in $K$ is equal to the closure of $D$ in $M$ because both sets are contained in $B \subseteq K \cap M$. By Claim~\ref{claim:odin},
$D^c = \overline{D}$. This makes $D^c$ a computable subspace of $M$.   Using Claim~\ref{claim:dva}, restrict $(C_n)_{n \in \omega}$ to a uniformly computable system $(C'_n)_{n \in \omega}$ of covers of $D^c$, witnessing that
$D^c$ is computably compact.
Since we have
$$x \in D = B(\xi, r)  \subseteq \overline{B(\xi, r)}  = B^c(\xi, r) \subseteq B(a,q),$$
 the computably compact ball $D^c = B^c(\xi, r) = \overline{B(\xi, r)}$ satisfies the required properties.\end{proof}


 \begin{definition}\label{def:compgrdef1} We say that a locally compact Polish group is
 \emph{computably locally compact} if  its domain is computably locally compact and  the group operations $\cdot$ and $^{-1}$ are computable  upon this domain.
 (The notion of a computably compact group is defined similarly.)
\end{definition}

In the discrete case, Definition~\ref{def:compgrdef} and Definition~\ref{def:compgrdef1} are actually equivalent; this follows quite easily from ~Fact~\ref{hgotGDDD112330OOFggrs}.
It is known that a profinite group is ``recursive profinite''  \cite{MetNer79, LaRo1,LaRothesis,Smith1} if, and only if, it is computably compact~\cite{EffedSurvey}. 
Further effective properties of locally compact groups have recently been   investigated  in more detail in~\cite{CoToGr,separ}.   For t.d.l.c.~groups,  the robust   notion of computable presentability~\cite{tdlc}  
is in fact equivalent to Definition~\ref{def:compgrdef1}. In Subsection~\ref{subsec:tdlc}, 
we will give further details on the result of~\cite{separ} establishing this, which is of importance in the present work. 
Given the evidence  that Definition~\ref{def:compgrdef1} is    natural and robust, we    choose it as the basic definition of computability for a locally compact Polish group.



\section{Computability of the 1-point compactification} \label{subs: 1-point compactification}
 The main result of the section  is Theorem~\ref{thm:1po}, an  algorithmic  version of taking the 1-point compactification of a   locally compact space.
  The presentation constructed in Theorem~\ref{thm:1po} will be used in Theorem~\ref{main:thm} and Theorem~\ref{thm:gr}.

\subsection{Computable $\sigma$-compactness, a strong form.}
Let $M$ be a Polish space with a fixed metric $d$.
\begin{definition}\label{def:uninbhd}
For $\epsilon>0$, we say that a set $U$ is a uniform $\epsilon$-neighbourhood ($\epsilon$-nbhd) of a set $V$ if
$$U \supseteq V(\epsilon) = \{x \in M: \exists y \in V \, d(y, x) < \epsilon \}.$$
In other words, for each point of $V$  together the whole  open $\epsilon$-ball around the point is contained in $U$.
\end{definition}
There is a danger of confusing computability-theoretic uniformity with topological uniformity (as defined above). In what follows next, we shall avoid saying `$U$ is a uniform $\epsilon$-neighbourhood of $V$' and instead we will use the associated 
notation $`U\supseteq V(\epsilon)$'. 

\

The technical lemma below appears to be  new  as stated, as  it is  stronger than   similar results in the literature; e.g., \cite{lc1}.
(Compare this lemma with  \cite[Proposition 8]{lc1} that also produces a nice list of compact neighbourhood  of a given space. Also, note that the notion of computable $
\sigma$-compactness as defined in  \cite[Definition 14]{lc1} appears to be weaker than the technical property established in the lemma below; see, e.g., \cite[Example 17]{lc1}.) 

\begin{lemma}[Strong computable $\sigma$-compactness]\label{lemma:sigma} Suppose $M$ is a computably locally compact 
Polish space.  There exists a nested sequence
$(K_n)_{n \in \omega}$ of uniformly computably compact sets $K_n \subseteq M$ and a uniformly computable sequence of positive reals $c_i \leq 2^{-i}$ with the following properties:
 \begin{enumerate}
 \item $K_{n+1} \supseteq K_n(c_n)$ for all $n$. (Recall Def.~\ref{def:uninbhd}.)   
 \item $M = \bigcup_{n \in \omega} K_n$.
 
 \item Each $K_n$  is represented as a finite union of computable closed balls.
 
 \item Given $x \in M$ we can uniformly effectively (in $x$) calculate some $n$ such that $x \in K_n$, and indeed  $x$ is inside one of the open balls whose respective closed balls make up $K_n$.
 
 \end{enumerate}

\end{lemma}

\begin{proof}
Suppose the computable local compactness of $M$ is witnessed by an operator $\Psi$. 
We can replace $\Psi$ by a functional $\Psi$ that turns a Cauchy name $(x_i)$ of a point to an enumeration of its name $N^x$ and then uses $\Psi$ to produce 
a compact name of the compact neighbourhood of the point. 
List all finite sequences of the form $\bar{x} = \langle x_0, \ldots, x_k \rangle$, where $x_i$ are special in $M$ and $d(x_i, x_{i+1})< 2^{-i}$, and calculate
the uniform sequence of those $\Psi^{\bar{x}}$ which halt with use at most  the length of $\bar{x}$.
Let $K_{\bar{x}}$ denote  (the computable index of) the computably compact set that is output by the procedure on input $\bar{x}$ if it halts.
Observe that any point $\alpha$ of $M$ (special or not) will be contained in some such  $K_{\bar{x}}$ that will be listed in the sequence.
Further, by Proposition~\ref{prop:ball} we may assume $K_{\bar{x}}$ is a basic computable closed ball $B^c$ with a uniformly computable radius; in the enumeration, 
it will be represented by its computable index.
Let  $(D_i)_{ i \in \omega}$ be the resulting sequence of (computable indices of) uniformly computably compact closed balls.

To make sure that all conditions of the lemma are satisfied, we need to modify the sequence as follows.
Set $K_0 = D_0$.
If $K_n$ has already been defined, uniformly fix a finite $2^{-n}$-cover $L_1, \ldots, L_k$ of $K_n$
so that the respective closed  balls $L^c_i$ are (uniformly) computably compact,
 and
set $$K_{n+1} =  D_{n+1} \cup \bigcup_{i} L^c_i.$$
(These balls are the first found balls in the sequence $(D_i)_{ i \in \omega}$; they must exist so we just search for the first found ones.)
It should be clear that conditions $(2)$, $(3)$, and $(4)$ of the lemma are satisfied by the sequence $(K_i)$ of uniformly computably compact subspaces.

It remains to define the parameters $c_n$ required in $(1)$.
In the notation above, use \cite[Remark 3.18]{EffedSurvey} to uniformly produce
a  finite cover of $K_n$ that formally refines $L_1, \ldots, L_k$.
This means that for each ball in this cover $B(a,r)$ there is some ball $L_j = B(b, q)$ containing it so
that $d(a,b) +r < q$. For each such $B(a,r)$ effectively choose a positive rational $l = l_j$ so that
$d(a,b) +r +l <q$ still. It follows from the triangle inequality that for any 
$z \in K_{n}$ we have that the $l$-ball around $z$ is contained in $L_j = B(b, q)$.
(If $w \in B(z, l)$ then $d(a,w) \leq d(a,b) + d(b,z) + d(z, w) \leq d(a,b) +r +l <q$.)
Choose $c_{n}$ to be  smaller than the least among all these $l_j$, $j =1, \ldots, k$, and also is at least twice smaller than $c_{n-1}$ (to make sure the sequence $(c_n)$ satisfies $c_n \leq 2^{-n}$).
Then we have that  
 $K_n(c_n)$ (see Def.~\ref{def:uninbhd}) 
 has to be included into $\bigcup_{i} L^c_i$, and thus
$$K_n (c_n) \subseteq  D_{n+1} \cup \bigcup_{i} L^c_i = K_{n+1}.$$
This gives $(1)$. \end{proof}

The next proposition, which can be readily deduced from the above lemma, was already mentioned in the introduction.  It appears to be new.  The proposition will not be directly used in the rest of the paper, however, it seems to be sufficiently important for the framework of computable Polish spaces.
\begin{proposition}\label{cor:HB}
Every computably locally compact Polish space admits a computable regular (i.e.~Heine-Borel) metric. Furthermore, $\delta$ and $d$ are effectively compatible, i.e.,
the map $Id: (M, d) \rightarrow (M, \delta)$ is a computable homeomorphism\footnote{In particular, the space remains computably locally compact under $\delta$.}.
\end{proposition}

\begin{proof} See Appendix~\ref{A}.
\end{proof}


\subsection{Effective 1-point compactification} 
Theorem~\ref{thm:1po} will be a key step in the definition of the $\Pi^0_1$-presentation of the Chabauty space of a computably locally compact group.

\begin{theorem}\label{thm:1po}  Given a computably locally compact Polish space $M$ that is not already compact, we can uniformly effectively produce a computable homeomorphic embedding
$f$ of $M$ into its computably compact 1-point compactification $M^* \cong M \cup \{\infty\}$.
Furthermore, $f^{-1}$ is computable (everywhere except for $\infty$), and this is also uniform.
\end{theorem}

\noindent \emph{A brief discussion.} 
In order to prove the theorem, we combine Lemma~\ref{lemma:sigma} with the metric on $M^*$ suggested in Mandelkern~\cite{point-paper}.  There are other potential ways to prove the theorem, e.g., using some effective version of the Urysohn's metrization theorem along the lines of \cite{metr}. However, the   explicit construction
 in  \cite{point-paper} provides us with some extra information about the constructed metric, and this will be   convenient in establishing the computable compactness of $M \cup \{\infty\}$. 
 We strongly suspect that the assumption that $M$ is not compact can be dropped without any effect on the uniformity of the procedure;  however, we will not verify this claim since this level of uniformity is not needed in the present paper.
 

\begin{proof}

Given a computably locally compact Polish space $(M, \rho)$, adjoin a point $\infty$ to $M$ and declare the point to be special in $M^*$.
Fix a sequence of uniformly computably compact neighbourhoods $(K_n)$ for $M$ and the corresponding computable sequence $(c_n)$ given by  Lemma~\ref{lemma:sigma}. 
For any $x \in M$, let $$h(x) = \sup \{c_i - \rho(x, K_i): i \in \omega \}.$$

\begin{claim}
The function $h \colon M \rightarrow \mathbb{R}$ is computable.
\end{claim}
\begin{proof}[Proof of Claim] Given $x$, we can uniformly find some $n$ so that $x \in K_n$; this is $(4)$ of Lemma~\ref{lemma:sigma}. 
It follows that, for such an $n$, $h(x) = \sup \{c_i - \rho(x, K_i): i  \leq  n \},$
because $(c_i)$ form a decreasing sequence and since $x \in K_m$ for all $m \geq n$.
Since $K_i$ are uniformly computable, $\rho(x, K_i)$ are  reals uniformly computable relative to $x$.
This makes $h(x)$ a uniformly $x$-computable real.
\end{proof}

Define $d$ on $M \cup \{\infty\}$ by the rule
 $$d(x,y) =  \begin{cases}
		 \inf \{ \rho(x,y), h(x)+ h(y) \}, \mbox{ if } x, y \in M, \\ 
		  h(x),  \mbox{ when } x \in M, y = \infty, 
			   \end{cases} $$
and  set $d(x,x) = 0$ for any $x \in M \cup \{\infty\}$.	
It is clear that for any $x,y \in M$ we have 
$$d(x,y ) \leq \rho(x,y),$$
and thus the identity map maps fast $\rho$-Cauchy names to fast $d$-Cauchy names. Thus, we can set $f = Id_M$.
It is shown in \cite{point-paper} that this definition of $d$ gives a metric compatible with the topology in the one-point compactification of $M$; the metric is clearly complete (by compactness).
By the claim above, $d(x,y)$ is uniformly computable for any pair of special points $x,y$ in $M \cup \{\infty\}$.

\

Recall that $K_n(c_n) \subseteq K_{n+1}$, where the latter is viewed as a computable Polish space w.r.t.~the new metric $d$. Recall that we set $f = Id_M$.
We can view $Id_M(K_n) = K_n$ as the computable image of a computably compact  space $K_n$ (w.r.t.~$\rho$) inside a computable
Polish space $K_{n+1}$ (w.r.t.~$d$).  It thus follows that $(K_n, d)$ is computably compact; see, e.g., \cite[Lemma 3.31]{EffedSurvey}, \cite[Theorem 3.3]{WeirauchSpacesStuff} and \cite[Proposition 5.5]{ArnoSurvey}.
This is also computably uniform (in $n$).
It follows from the definition of $d$ that, when $x \in ( M \cup \{\infty\} ) 
\setminus K_{n+1}$, we have $$\rho(\infty, x) = h(x) \leq c_{n+1} < c_{n},$$
see \cite{point-paper} for an explanation and for further details. 
We conclude that
$$M \cup \{\infty\}   = B(\infty, c_n) \cup \bigcup_{i \leq n+1} K_i.$$
To find a finite $2^{-n}$-cover of $M \cup \{\infty\}  $,  recall that $c_n < 2^{-n}$. Effectively uniformly fix a finite $2^{-n}$ cover  of  $\bigcup_{i \leq n+1} K_i$ (in the new metric $d$).
Together with $B(\infty, c_n)$, this gives a finite $2^{-n}$-cover of $M \cup \{\infty\}$.

\

We now show that $f^{-1}$ is also computable on $M$. (Recall that $f$ is the identity map, but viewed as a map between two different metrizations of $M$.) Assume $x \neq \infty$.
Calculate $d(x, \infty)$ to a precision sufficient to find some $n$ so that $d(x, \infty)> c_n$.
By the definition of the metric $d$ and the sequence $(c_n)$, it must be that $x \in K_n$. 
Since $f$ is  just the identity map, it is clearly guaranteed that $f^{-1}(x) \in (K_n, \rho)$.
It is well-known that the inverse of a computable homeomorphic map between computably compact spaces is computable;  see, e.g.,  \cite[Theorem 3.33]{EffedSurvey} or \cite[Corollary 6.7]{BBP}. This is also clearly effectively uniform. Apply this fact to $f: (K_n, \rho) \rightarrow (K_n, d)$ to calculate $f^{-1}(x)$. 
\end{proof}

Let $M$ be a computably locally compact space, $M^* = M \cup \{\infty\}$ be its computably compact 1-point compactification, and $f: M \rightarrow M^*$ a computable embedding having computable inverse; see Theorem~\ref{thm:1po}.
 
 The Hausdorff  distance between finite sets of computable points  is clearly computable. The hyperspace  of compact subspaces of a computably compact space is itself computably compact; this is folklore. (To see why, note that, given a finite $\epsilon$-cover, we can restrict ourselves to finite subsets of the centres of the balls involved in the cover. These finite subsets will give rise to an $\epsilon$-cover of the hyperspace.)
  Since $M^*$ is computably compact, we conclude:

 \begin{fact}\label{fact:ccc} For a computably locally compact Polish space $M$, the space  $\mathcal K(M^*)$ is computably compact as well, and this is uniform.
 
 \end{fact}
\noindent In a different framework and using different methods, a similar fact was established in~\cite[Subsection 5.2]{lc1}.

\section{Proof of Theorem~\ref{main:thm}(1): $\mathcal{S}(G)$ is $\Pi^0_1$  in $\mathcal{K}(G^*)$.}

The first subsection offers an explanation of the fact, discussed in the introduction, that  the Chabauty space of a l.c.~$G$ can be viewed as a closed subspace of the hyperspace of closed sets of the 1-point compactification of $G$. 
In   the subsequent Subsection~\ref{ss:S(G) sub},  we will use the results from the foregoing section, specifically  Fact~\ref{fact:ccc}, to prove that this presentation is indeed $\Pi^0_1$.

\subsection{The Chabauty space of $G$ and the $1$-point compactification of $G$} \label{subs:folklorestuff}

The fact below appears to be folklore   in topological group theory.

\begin{fact}\label{fact:ChAl} Let $G$ be locally compact Polish group. There is a homeomorphic embedding for the Chabauty space $\mathcal S(G)$ into the hyperspace $\mathcal K(G^*)$
of closed (thus, compact) subsets of the $1$-point compactification $G^*$ of $G$. The embedding is given by the map $i:C \mapsto C \cup \{\infty\}$.
\end{fact}

 We are not aware of a reference where this fact is thoroughly explained\footnote{
Cornulier mentions this fact  at the very beginning of his paper~\cite{Cornulier:11}, but in terms of convergence of nets.
A more general fact  appeared   in Fell~\cite{Fell}; see (II) on p.~475.
It is stated much more explicitly, but still not quite in the form that we need it,  in  lecture notes  by
   de la Harpe on arxiv~\cite{delaharpe};  see ``2.~Prop.~(v)'' on pages 2-3.}, so we give its detailed verification in Appendix~\ref{B}.


\subsection{$\mathcal S(G)$ as a $\Pi^0_1$-subset of $\mathcal K(G^*)$} \label{ss:S(G) sub}Recall that an open set $U$ is computably enumerable~if  there is a computably enumerable sequence of basic open balls making up the set. A closed set is $\Pi^0_1$ or effectively closed if its complement is computably enumerable open. 
For a computably l.c.~$G$,
we think of  $\mathcal K (G^*)$ as given by  its computably compact presentation established in Fact~\ref{fact:ccc}.

\begin{proposition}\label{thm:cha} For a computably locally compact group $G$, the Chabauty space $\mathcal S(G)$  is $\Pi^0_1$ (effectively closed) in the computably compact hyperspace $\mathcal K(G^*)$ of compact subsets of the 1-point compactification of $G$.

\end{proposition}

\begin{proof}
 In $\mathcal K(G^*)$,  one  can distinguish between special points that contain $\infty$ and those that are generated by points coming from $G$;
 recall that the point $\{\infty\}$ is special in $G^*$. 
 We  suppress the subscript ``$H$'' in ``$d_H$'' throughout and simply write  ``$d$'' for the Hausdorff metric.  
Recall   that in Theorem~\ref{thm:1po}, $f : G \rightarrow G^*$ is computable, and $f^{-1}$ is computable  everywhere   except for $\infty$ where it is undefined.

A non-empty closed set  $H \subseteq G$ is a closed subgroup iff $x y^{-1} \in H$, for all $x,y \in H$.
For the corresponding $K = H \cup \{\infty\}$ in $\mathcal{K}(G^*)$, this condition corresponds to $$\forall x, y \in K \, \, \min \{ d(x, \infty), d(y, \infty) \} > 0 \rightarrow  f( f^{-1}(x)\cdot  [f^{-1}(y)]^{-1} )\in K, $$
for all $x,y \in K \setminus \{\infty\}.$
Recall that $f$ is effectively continuous and effectively open, and the group operations are effectively continuous and effectively open as well.
It follows that the function $$g(x,y)=  f( f^{-1}(x) \cdot [f^{-1}(y)]^{-1} )$$
is computable on its domain.
Thus, the condition above \emph{fails} if, and only if, for some small enough basic balls $B, D, V$ (in the metric of $G^*$), all not containing $\infty$,
we have that 
$$ [B \cap K \neq \emptyset]  \wedge [D \cap K  \neq \emptyset]  \wedge [V \cap K  = \emptyset] \wedge [g(B, D) \subseteq V],$$
where the basic open balls not containing $\infty$ and satisfying $g(B, D) \subseteq V$ can be effectively enumerated.
We will say that a triple of basic open balls $(B, D, V)$ is a  \emph{counterwitness} if  $g(B, D) \subseteq V$.   

We are ready to effectively enumerate the complement of $\mathcal S(G)$ in $\mathcal K(G^*)$.
Given a counterwitness triple $(B, D, V)$, define an effectively open set in $\mathcal K(G)$ that consists of all  closed (compact) subsets $K$
that satisfy:

\begin{enumerate}

\item $B \cap K \neq \emptyset $;

\item  $D \cap K \neq \emptyset $;

\item $d(K, cntr(V))> r(V)$,

\end{enumerate}
where $cntr(V)$ and $r(V)$ are the distinguished rational centre and the radius of the basic open $V = B(cntr(V), r(V))$
in the metric of $G^*$. (Note that the former two intersections, if they hold, will be witnessed already by a sufficiently close finite approximation to $K$, in $\mathcal{K}(G^*)$, that is, by a finite collection of special points in $G^*$. The same can be said about condition $(3)$. In other words, the conditions are $\Sigma^0_1$.) Let $\mathcal{W}$ be the collection of all such $K$, where $(B, D, V)$ ranges over all counterwitness triples. Since $(1)-(3)$ are uniformly $\Sigma^0_1$, $\mathcal{W}$ is effectively open in $\mathcal{K}(G^*)$.

If a compact set does not contain $\{\infty\}$, then it is separated from $\{\infty\}$ by a non-zero distance. Also, any subgroup of $G$ has to contain the identity element. It is clear that
$$\mathcal{V} = \{ K \in \mathcal K(G^*) : d(\infty, K) >0 \mbox{ or  } d(f(e), K) > 0 \} $$
is an effectively open set as well.

Then $K \in \mathcal K (G^*)$ corresponds to a closed subgroup of $G$ if, and only if,
there is no counterwitness triple for $K$, and $f(e),\infty \in K$. It follows that 
$$ \mathcal S(G) = \mathcal K(G^*) \setminus ( \mathcal{V} \cup \mathcal{W}),$$
where $\mathcal{V} \cup \mathcal{W}$ is a c.e.~open set in $\mathcal K(G^*)$.
We conclude that  $\mathcal S(G)$ is  $\Pi^0_1$ in   $\mathcal K(G^*)$. 
\end{proof}

\section{Proof of Theorem~\ref{main:thm}(2): \\  $\mathcal{S}(G)$ is not computably closed in general}

We will show more than is stated in Theorem~\ref{main:thm}(2).  Recall that that for a $\Pi^0_1$-closed subset of a computable space, being computable is equivalent to having a computable dense sequence of points; see Subsection~\ref{subs:prelimPolish}.   

We will prove that  there exists a computable abelian discrete $G$ so that the \emph{only} computable points of $\mathcal S(G)$
are the ones corresponding to  $\{0\}$ and $G$, yet $G$ has uncountably many  subgroups. Clearly,  $\{0\}$ and $G$ have to be computable subgroups; in this sense our result is also optimal. 
We split the proof into several subsections.

\subsection{The effective correspondence lemma} In the notation of the previous section, we have:

\begin{lemma}\label{lem:points}
Let $G$ be a computably locally compact group, and $\mathcal S(G) \subseteq \mathcal K(G^*)$ its effectively closed Chabauty space.
Then the following are equivalent:

\begin{enumerate}
\item $H \leqq_c G$ is computably closed.
\item $H^* = H \cup \{\infty\}$ is a computable point in $\mathcal K(G^*)$.
\end{enumerate}
\end{lemma}

\begin{proof}
The proof is not difficult, but it relies heavily on various properties of c.e.~closed and open sets, computable and effectively open maps (see Subsection~\ref{subs:prelimPolish}) and the technical results established earlier (e.g., Theorem~\ref{thm:1po}).
We will also use the following, apparently   well-known, fact. 

\begin{fact}
A closed subset $C$ of a computably compact space $K$ is computably closed if, and only if, 
$C$ is a computable point in the space $\mathcal{K}(K)$ (under the Hausdorff metric).
\end{fact}

\begin{proof}
We have that $C$ is computably closed in $K$ if, and only if,   it is a computably compact subspace of $K$; see, e.g., \cite[Proposition 3.29]{EffedSurvey}.

Assume $C$ is computably closed. Since $C$ is computably compact,
we can uniformly compute the Hausdorff  distance 
$d(x, K) = \inf_{y \in C} d(x,y)$ to any special point $x \in K$. (This is  because taking the infimum of a computable function over a computably compact set 
gives a computable real, and this is uniform, as mentioned in Subsection~\ref{subs:prelimPolish}.) If follows that we can also uniformly calculate the Hausdorff distance between $C$ and any finite set of special points.
The finite sets of special points of $K$ are the special points in $\mathcal{K}(K)$. Given $n$, search for a finite set of special points $D_n$ at distance $2^{-n}$ from $C$.
The sequence $(D_n)_{n \in \omega}$ is a computable fast Cauchy name of $C$.

Conversely, suppose $C$ is a computable point in $\mathcal{K}(K)$. A basic open ball $B(x, r) \subseteq  K$ does not  intersect $C$ if (and only if) $d(x, C) > r$, which is $\Sigma^0_1$.
On the other hand, $C\cap B(x,r) \neq \emptyset$ is equivalent to $d(x, C) < r$, which is also $\Sigma^0_1$. The condition ``$C\cap B(x,r) \neq \emptyset$ is $\Sigma^0_1$''  is equivalent to $C$ having a computable dense sequence of points (folklore; see \cite[Lemma 3.27]{EffedSurvey}).\end{proof}

Assume (1).
By the fact above, it is sufficient to show that $H^*$ is computably closed.  Fix $f$ given by Theorem~\ref{thm:1po}.
 Since $f: G \rightarrow G \cup \{\infty\}$ is computable, it uniformly maps computable points to computable points.
In particular, it maps the dense computable set $X$ of $H$ to a uniformly computable sequence that we denote $f(X)$. The set $f(X) \cup \{\infty\}$ is dense in $H^*$.  It follows that $H^*$ is c.e.~closed.
To conclude that $H^*$ is computably closed, we need to show that it is $\Pi^0_1$, i.e., its complement is c.e.~open. By Theorem~\ref{thm:1po}, $f^{-1}$ is computable when defined. Since $f$ is a homeomorphism of $G$ onto its image in $G\cup \{\infty\}$ (and in particular is injective), this is equivalent to saying that $f$ is effectively open.
Since $f$ is effectively open, it maps the c.e.~complement of $H$ (in $G$) to a c.e.~open set. But this set is  the complement of $H^*$ in $G^*$. Thus, $H^*$ is computable closed in  
$G^*$, and 
 (2) follows from (1).

Now assume $(2)$. 
As noted above, $(2)$ is equivalent to saying that $H^*$ is computable closed in $ G^*$. 
Recall that  $f^{-1}$ is computable (and defined) everywhere except $\infty$.
Since $\infty$ is a special point in $G^*$, we can effectively list those points in the computable dense sequence of $H$ that are not equal to $\infty$.
 The $f^{-1}$-images of these points make up an effective dense set of 
$H$ in $G$. It shows that $H$ is c.e.~closed.
Let  $U$ be the c.e.~complement of $H^*$ in $\mathcal K(G^*)$; note $\infty \notin U$. Since $f$ is computable it is effectively continuous, 
and thus $f^{-1}(U)$ is c.e.~open in $G$. Since we have $H = G \setminus f^{-1}(U)$, it follows that $H$ is both c.e.~and $\Pi^0_1$, and therefore computably closed.
\end{proof}

We view a computable discrete group  as a computable Polish group w.r.t.~the discrete metric: 
\begin{center}
 $d(x,y) =1$ whenever $x \neq y$.
 \end{center}
 Any such computable discrete group is clearly computably locally compact. Furthermore, a subset of such a group is c.e.~iff it is c.e.~as a closed or open subset of~$G$.
 In particular, we have:
\begin{fact} \label{fact:fff} For a discrete computable $G$,  a subgroup $H \leq G$ is a computable subgroup iff it is computably closed w.r.t.~the discrete metric defined above. 
\end{fact}

By Fact~\ref{fact:fff} above, in the discrete case we can   `forget' about topology and simply work in the standard, discrete recursion-theoretic setting.
This approach is taken in the remainder of this section, which does not mention  computable topology  at all.

\subsection{Computably simple  groups}

All groups in this subsection are discrete, abelian, and at most countable. A lot is known about  computable presentations of such groups; however, the definition below is new. (For the foundations of computable abelian group theory, we cite the surveys \cite{Khi,melbsl}.)

\begin{definition}
A computable group is   \emph{computably simple} if it has no non-trivial computable proper normal subgroups.


 


\end{definition}


\begin{theorem}\label{thm:notsimple}
There exists a computably simple presentation of the free abelian group of rank $\omega$.
\end{theorem}

\begin{proof}
We build a computable presentation $G$ of the free abelian group of rank $\omega$.
We will initially begin with building $G$ upon the basis $B = (b_i)_{i \in \omega}$ which freely generates 
$G$. However, at a later stage we may declare $b_i \in  \langle  b_j : j< i \rangle $ for some $i$.
We will make sure that for an infinite set of indices $I \in 
\Delta^0_2$, $\{b_i : i \in I\}$ remains independent, and so
$G$ is freely generated by  $\{b_i : i \in I\}$.

\subsubsection{Introducing new relations}\label{sss:sss} Before we proceed to the requirements and explain the basic strategy,
we briefly explain the folklore technique in computable abelian group theory that can be traced back to Mal'cev~\cite{Mal}. The technique
was modified and used in many results including, e.g., \cite{Nurt, dobr:81, HMM, GLSOrdered, mel1} and, more recently, by the authors in \cite{lupini}.
We will need only the most basic version of this technique.

In the notation of the proof, at a stage $s$ we will have a finite partial group $G_s$ consisting of
linear combinations of the form 
$$\sum_{i \leq s} m_i b_i, $$
where $| m_i | \leq s.$ The operations $+$ and $-$ on these linear combinations will be declared naturally as well, assuming the result is also a linear combination of this form.
(If not, we delay the definition until a later stage.)  We need to also declare which linear combinations are equal (and which are not equal) in $G_s$.

At  stage $s+1$, according to the instructions  of our strategies,  we may want to declare $b_n \in \langle b_i: i \neq n \rangle $ for some $n \leq s$.
To simplify notation, assume $n=s$ and furthermore we need to set $b_n \in \langle x \rangle$ for some fixed non-zero $x \in  \langle b_i: i < s \rangle \cap G_s$.  We also assume that no further relations are imposed on $b_i, i <s$. 

The potential tension is as follows.
We would like to identify the elements of $G_s$ with an initial segment of $\omega$, so that  different indices correspond to unequal elements in $G_s$ and, indeed, in $G$.  For that, we must make sure that 
the inclusion  $$G_s \subseteq G_{s+1}$$ induces an isomorphic embedding of partial structures. Only in this case we can safely set $G = \bigcup_s G_s$ and claim that $G$ is computable.
For example, if  $b_s \neq 3x$ in $G_s$ then we cannot possibly   declare  $b_s = 3 x$ in $G_{s+1}$. Also, in $G_s$, we cannot declare any relation that would imply $b_s = 3 x$.
(Equivalently, we must make sure that $G$ is isomorphic to the quotient  of the free abelian group upon $(b_i)_{i \in \omega}$ by its computable subgroup, and not merely a c.e.~subgroup.)

To circumvent this difficulty, we fix a very large natural number $M$ greater than any number seen so far in the entire construction and declare
$$b_s = M x.$$

Note that $G_s$ does not include any linear combination with coefficients so large. Since $M$ is very large, we claim that
$$\sum_{i \leq s} m_i b_i   = \sum_{i \leq s} n_i b_i   \iff \sum_{i < s} m_i b_i  + m_s M x = \sum_{i < s} n_i b_i + n_s M x,$$
whenever $|n_i|, |m_i| \leq s$. The left-to-right implication is obvious. 
To see why the converse implication $\Leftarrow$ holds, let $$x = \sum_{i < s} k_i b_i,$$
where $|k_i| \leq s$ and at least one coefficient $k_i \neq 0$; w.l.o.g.~$i=0$. (This is where we use that $x \neq 0$.) We have that 
$$\sum_{i < s} (m_i   + m_s   k_i M) b_i =  \sum_{i < s} (n_i   + n_s   k_i M) b_i.$$
Since we assumed no further relations hold upon $b_i, i<s$, it must be that for all $i<s$,
$$m_i   + m_s   k_i M = n_i   + n_s   k_i M.$$
Since $M$ is very large,  $m_i = n_i $ for all $i < s$, and thus in particular $m_0 = n_0$.
It follows that $m_s k_0 M = n_s k_0 M$, and since $k_0 \neq 0$, we arrive at $m_s = n_s$, showing~$\Leftarrow$.

In summary, by setting $b_s = M x$, where $M$ is very large, we can always guarantee that the finite part (of the open diagram of) $G_s$ of  $G$
built by stage $s$ is preserved. Of course,  the procedure described above can be  iterated
for other $x'$ and $b_j$.

\subsubsection{The requirements} We identify the domain of $G$ with the set of natural numbers, and we identify computable subsets of $G$
with their characteristic functions. Let $H_e$ be a (potential) computable subset of $G$ with characteristic function $\varphi_e$.
We must meet the following requirements:

\

$R_{e}: \,\, H_e \mbox{ does not represent a proper subgroup of } G. $

\

\noindent (Of course,  ``$H_e$ represents a proper subgroup of $G$'' means that $\varphi_e$ is total and is a characteristic function 
of a subgroup of $G$ which is neither $\{0\}$ nor is equal to $G$.)

\subsubsection{The strategy for $R_{e}$} 
Recall that $(b_i)_{i \in \omega}$ are generators of the group; furthermore this set contains
a subset that freely generates the group. Without loss of generality, we may assume that the neutral element $0$ of $G$ has index $0$ in the indexing of the elements by natural numbers.
The strategy goes through the following substeps. 

\begin{enumerate}
\item Wait for some nonzero $x \in H_e$ and $y \notin H_e$. (If this never happens, either $H_e = G$ or $H_e$ is partial.)

\item Pick $b_j, b_k$ with thier indices $j, k>3e$ larger than any number seen so far in the construction, and check whether $b_j, b_k\in H_e$.
Assuming $\varphi_e$ converges, there are two cases:

 \begin{enumerate}
       \item[Case 1]  $b_j \notin H_e$ or $b_k \notin H_e$. Say, $b_i \notin H_e$.  Declare $b_j = n x$, where $n$ is larger than any number used in the definition of $G$ so far.
       (Since  $x \in H_e$ in $(2)$,  we have diagonalised.)

         \item[Case 2]  $b_j, b_k \in H_e$. In this case pick $n$ larger than any number used in the definition of $G$ so far, and pick $m$ much larger than $n$
         so that $GCD(m,n) =1$. Declare $ b_j = n y$ and $b_k = m y$.     (Since $GCD(n,m)=1$ it has to be that $y \in  \langle m y, n y \rangle$ and $y \in H_e$; we have diagonalised.)
        
\end{enumerate}

        \end{enumerate}

Note that in each case,  we have guaranteed that $R_e$ is met.  As explained in Subsection~\ref{sss:sss}, the choices of $m$ and $n$ guarantee that the diagram of 
the finite part of $G_s$ is preserved when we introduce new relations in (2).



\subsubsection{Putting the strategies together} In the construction, we keep building $G = \bigcup_{s \in \omega} G_s$ to be $ \langle b_i : i \in \omega \rangle$, taking into account the 
relations on the generators $b_i$ which can be introduced by the strategies. At a stage $s$, $G_s$ consists of sums of the form $\sum_i m_i b_i$; we can arrange 
the indices so that any any such sum, $|m_i| \leq s$ and $i \leq s$. 

We arrange the strategies into a priority order, so that $R_{e}$ is declared higher priority than $R_{e'}$ if $e < e'$. 
If at some stage $R_e$ declares some element $b_j$ or $b_k$  dependent on some other elements (according to its instructions in (2)), then we say that $R_e$ \emph{acts}
at the stage; in this case we also initialise all strategies of weaker priority that $R_e$ by setting all its parameters undefined.  By induction, we obtain that every $R_e$ can be initialised only at most  finitely many times, and thus each $R_e$ is eventually met.
The choice of indices $i, j > 3e$ (which are also additionally chosen bigger than any number seen so far)  guarantees that at least $e$-many of the elements $\{b_s : s \leq 3e\} $ will never be declared dependent on each other. This is because at most $e$-many strategies can have their witnesses $b_j$ and $b_k$ with $j, k < 3e$; later, if they are initialised, their witnesses will
be chosen too large. 
Consequently, we have that the rank of $ \langle b_i : i \in \omega \rangle$ is infinite. Also, we never introduce any elements to $G$ that would not be generated by $(b_i)_{i \in \omega}$, and thus
$G = \langle b_i : i \in \omega \rangle$ is free abelian of rank $\omega$ freely generated by some subset of $(b_i)_{i \in \omega}$.
 \end{proof}

\subsection{Finalizing the proof of Theorem~\ref{main:thm}(2)} By Theorem~\ref{thm:notsimple}, there exists a computable presentation $G$ of the free abelian group of countably infinite rank that is computably simple. Clearly, it has uncountably many proper subgroups. Fact~\ref{fact:fff} implies that the group can be viewed a computably locally compact that has no non-trivial computable closed subgroups. By Lemma~\ref{lem:points}, its Chabauty space $\mathcal S(G)$ will have uncountably many points, but only two computable points, specifically $G^*$ and $\{0\}^*$. Thus, $\mathcal S(G)$ does not have a computable dense sequence (in $\mathcal K(G^*)$). This finishes the proof of Theorem~\ref{main:thm}(2).

We have constructed a \emph{fixed} computable presentation $G$ of the free  abelian group $\mathbb{Z}^{< \omega}$ so that $\mathcal S(G)$ is not computable in $\mathcal K(G^*)$.
Clearly, the `natural' computable presentation of this group does not have this pathological property. 
We leave open whether    there exists  a computable discrete group $G$ so that for \emph{each} computable presentation $H$ of $G$,
$\mathcal S(H)$ is not computable in $\mathcal K(H^*)$.
We also wonder if  there exists a reasonable description of computable discrete abelian groups that admit computably simple presentations.


\section{The Chabauty space of a t.d.l.c.\ group via its  meet groupoid}\label{sec:mimi}

In this section we establish an effective correspondence between the Chabauty space $\mathcal S(G)$  of $G$ and the meet groupoid $\mathcal{W}(G)$ of all compact open subsets of a totally disconnected locally compact (t.d.l.c.) group $G$, which is defined below.

{Recall that a groupoid   is  given by a domain $\+ W$ on which     a unary operation $(.)^{-1}$ and a partial binary operation, denoted  by ``$\cdot $", are defined. These operations satisfy the following conditions:
 \bi \item[(a)]  $(A \cdot B)\cdot C= 
 A \cdot (B\cdot C)$,  with either both sides or no side defined;  \item[(b)]  $A\cdot A^{-1}$ and $A^{-1}\cdot A$ are always defined; \item[(c)] if $A\cdot B$ is defined then $A\cdot B\cdot B^{-1}=A$ and $A^{-1}\cdot A\cdot  B =B$.\ei } 
 
 {A \emph{meet groupoid}~\cite{tdlc} is a groupoid  $(\+ W, \cdot , {(.)}^{-1})$ that is also a meet semilattice  $(\+ W, \cap  ,\ES)$ of which  $\ES$ is the  least element.     Writing $A \sub B \LR A\cap B = A$ and letting the operation $\cdot$ have preference over $\cap$,
it satisfies the conditions   \bi \item[(d)]   $\ES^{-1} = \ES = \ES \cdot \ES$,    and  $\ES \cdot A$ and $A \cdot \ES$ are undefined for each $A \neq \ES$,  
\item[(e)]   if $U,V$ are idempotents such that $U,V \neq \ES$, then   $U  \cap V \neq \ES$,  \ei

 \bi \item[(f)] $A \sub B \LR A^{-1} \sub B^{-1}$, and

\item[(g)]   if  $A_i\cdot B_i$ are defined ($i= 0,1$) and $A_0 \cap A_1 \neq \ES \neq B_0 \cap B_1$, then \bc $(A_0  \cap A_1)\cdot (B_0 \cap B_1) =  A_0 \cdot  B_0 \cap A_1 \cdot B_1 $. \ec   
 \ei}

\begin{definition}\cite{tdlc} Let $G$ be a t.d.l.c.\  group.  We define the  meet groupoid  $\+ W(G)$ of $G$, as follows.  Its  domain consists of  the compact open  cosets in $G$ (i.e., cosets of compact open subgroups of $G$), as well as the empty set.  
 We define $A\cdot B$ to be the usual product $AB$ in case that  $A= B= \ES$, or $A$ is a left  coset of a subgroup $V$ and $B$ is a right  coset of $V$; otherwise $A \cdot B$ is undefined.  \end{definition}
 Notice that the groupoid defined above also carries a natural partial order $\subseteq$ of set-theoretic inclusion under which it is also forms a meet semilattice, hence the name.
It has been established in \cite{tdlc} that a 
 duality holds between t.d.l.c.~groups and their respective meet groupoids. This duality is also fully effective in the sense that
it gives a computable $1$-$1$ correspondence between the respective computable meet groupoids and computably t.d.l.c.~groups.

\subsection{Computable setting}\label{subsec:tdlc} A separable t.d.l.c.~group is homeomorphic to a locally compact subspace of $\omega^{\omega}$. It therefore makes sense to represent such groups using trees.
Our (rooted) trees are viewed as sets of strings in $\omega^{ <\omega}$ closed under prefixes. Finite strings can be viewed as `nodes' of the tree.
A tree has no dead ends if every finite string on the tree is extendible to an infinite path through the tree.
The space of paths through a tree $T$ is denoted $[T]$. The space ${T}$ is a metric space under the shortest common initial segment ultrametric.
A computable tree $T$ with no dead ends evidently induces a computable Polish presentation on~$[T]$.

\begin{definition}\label{def:nicelc}
We say that a computable tree is \emph{nicely} computably locally compact if it has no dead ends, only its root can be $\omega$-branching, and
given a node 
one can   compute the number of its successors.
\end{definition}

Clearly this implies that  $[T]$ is a computably locally compact space.
\begin{definition}[\cite{tdlc}]\label{def:main1}
       Let $G$ be a Polish t.d.l.c.~group. A \emph{computable Baire presentation} of $G$
    is a topological group $\hat G \cong G$ of the form  $\hat G= ([T], \cdot, ^{-1})$ such that
  \begin{enumerate} \item  $T$ is a nicely effectively  locally compact tree;
  \item  $\cdot\colon [T] \times [T] \to [T]$ and $^{-1} \colon [T] \to [T]$ are computable (as operators).
 \end{enumerate}
  \end{definition}

We have already mentioned the following  result:

\begin{theorem}[M. and Ng~\cite{separ}, Thm.\ 1.2.] \label{thm:mn} For a t.d.l.c.\ group~$G$, the following are equivalent:
\begin{enumerate}
\item $G$ has a computably locally compact presentation, and

\item  $G$ has a computable Baire presentation.

\end{enumerate}

\end{theorem}
The implication $(2) \rightarrow (1)$ in the theorem above is of course obvious. It is important for us that $(1) \rightarrow (2)$ is witnessed by  a computable Baire~presentation of $G$ that is \emph{computably homeomorphic} to the given computably locally compact copy of the group. (Recall that  this in particular  means both the homeomorphism $f$ and its inverse $f^{-1}$ are computable.) Thus, the closed subgroups of these presentations will be in a $1$-$1$ effective correspondence induced by this effective homeomorphism. 
It follows that, in the special case of t.d.l.c.~Polish groups we can use the rather convenient Definition~\ref{def:nicelc} above in place of the  (seemingly) more general notion of a computably locally compact presentation. In particular, for a fixed computable tree with no dead ends, the notions of a computable open and closed sets become rather explicit:

\begin{definition}  An open subset $R$ of $[T]$ is called computable if $\{\sigma \in T \colon \, [\sigma] \sub R\}$ is computable. 
 A closed subset $S$ of $[T]$ is called computable if $[T]-S$ is computable. Equivalently,   the  subtree $\{\sigma \in T \colon \, [\sigma] \cap H \neq \ES\}$ corresponding to $S$ is computable. \end{definition} 

For the dual meet groupoid, we shall use the following:

 \begin{definition}\cite{tdlc} A meet groupoid $\+ W$ is called \emph{Haar computable}~if 
\bi \item[(a)]  its  domain is (indexed by) a computable subset $D$ of $\NN$; 

\item[(b)] the   groupoid and meet operations    are computable; in particular, the relation $\{ \la x,y \ra\colon \, x,y \in D \lland x\cdot y \text{ is defined}\}$   is computable;
 
 \item[(c)]     the partial function with domain contained in  $D \times D$ sending  a pair of subgroups $U, V\in \+ W $ to $|U:U\cap V|$ is   computable.   \ei
    \end{definition}

The key result relating the definitions above is the following:

 \begin{theorem}[\cite{tdlc}]  \label{thm:ma}  A group   $G$  has a computable  Baire~presentation if, and only if, its  meet groupoid $\mathcal{W} (G)$ has a Haar computable copy. 
\end{theorem} 

This  effective duality result also enjoys a number of further effective properties. 
For instance, the elements of the Haar computable copy of $\mathcal{W} (G)$ produced based on a Baire presentation of $G$ are in an effective $1$-$1$ correspondence with 
computable compact open cosets of the respective computable  Baire~presentation, and vice versa.
Note however that $\mathcal{W} (G)$ captures only \emph{compact} open subgroups. What about closed subgroups of $G$? How are they reflected in $\mathcal{W} (G)$? We answer this question (non-effectively) in the following section, and then in the subsequent section we establish an effective correspondence in order to prove Theorem~\ref{thm:gr}.



\subsection{Ideals in a meet groupoid}

   
 
 
 

To this end,  $G$ will always denote a t.d.l.c.\ group, and $\mathcal W= \+W(G)$ its meet groupoid. 
Recall that the compact open cosets form a basis for the topology of $G$.  For $A, B_1, \ldots, B_n \in \mathcal W$, the relation $A \sub \bigcup_iB_i$ is  first-order definable in    $\mathcal W$,. To see this, note that   its complement is given by the formula 
  $\exists C \sub A  \bigwedge_i [C \cap B_i= \ES]$. So in $\mathcal W$ as an abstract structure we can define its set of ideals in the following sense.

\begin{definition}\label{def:id} We say that a set $\mathcal J \sub \mathcal W$ is an \emph{ideal} of $\mathcal W$ if  $\ES \in \mathcal W$ and \bc $(B_1, \ldots, B_n \in \mathcal J \lland A \sub \bigcup_iB_i)$ $\RA$  $A \in \mathcal J$. \ec \end{definition} 
\n Similar to Stone duality between Boolean (Stone) spaces and Boolean algebras, open subsets $R$ of $G$ naturally correspond to ideals $\mathcal J$  of $\mathcal W$ via the maps

\bc $R \mapsto  \{ A \colon \, A \sub R\}$ and $\mathcal J \mapsto  \bigcup \mathcal J$. \ec
We also  write  $S_\mathcal J = G - \bigcup \mathcal J$, which is   a closed set in $G$.

\begin{definition} \label{df:ideals} Let $\mathcal J$ be an ideal closed under inversion $A \mapsto A^{-1}$.

\n (i)  We say that   $\mathcal J$ is an \emph{open-subgroup ideal} if  
\bc $A , B \in \mathcal J \lland  A\cdot B$ defined  $\RA$ $A \cdot B \in \mathcal J$. \ec

\n (ii) We say that   $\mathcal J$ is a \emph{closed-subgroup ideal} if 
\bc $A\cdot B$ defined $\lland A\cdot B \in \mathcal  J$ $\RA$ $A \in \mathcal J  \lor B \in \mathcal J$. \ec
  \end{definition} 
Note that the set $\mathcal {CSI}(\mathcal W)$ of closed-subgroup   ideals   is closed in the product topology on $\mathcal P(\mathcal W)$, and hence totally disconnected.  It is immediate that  $\mathcal J$ is an open-subgroup ideal if, and only if,   $R_\mathcal J$ is an (open)  subgroup of~$G$. For closed-subgroup ideals, the analogous fact holds:
\begin{lemma}\label{lemma:Sj} $\mathcal J$ is closed-subgroup ideal $\LR$ $S=S_\mathcal J$ is a closed  subgroup of $G$. \end{lemma}  \begin{proof} \lapf Immediate. 

\rapf Let $g,h \in G$. If $g^{-1} \not \in S$ then there is $A\in \mathcal J$ such that $g^{-1} \in A$. Since $\mathcal J$ is closed under inversion, this implies that $g \not \in S$.  

If $g, h \in S $ but $gh \not \in S$ then pick $C \in \mathcal J$ such that $gh \in C$. By continuity there are $A, B\in \mathcal W$ such that $g\in A, h \in B$ and $AB \sub C$ (the   product of subsets in $G$). Let $A$ be left coset of $U$ and $B$ be right coset of $V$. Let $L = U \cap V$. Replacing $A$ by $\hat A = gL$ and $B$ by  $\hat B = Lh$, we have $\hat A \cdot \hat  B \sub C$ and hence  $\hat A \cdot \hat  B \in \mathcal J$. But $\hat A \cap S \neq \ES \neq \hat B \cap S$, so neither $\hat A$ nor $\hat B$ are in $\mathcal J$, contrary to the assumption that $\mathcal J$ is a closed-subgroup ideal.
  \end{proof}

We now relate   the Chabauty space of a t.d.l.c.\ group~$G$ to its meet groupoid.
Using the notation of \cite[Section 2]{Cornulier:11}, a basic open set in $\mathcal S(G)$ has the form 
\begin{equation} \label{eqn: BO} \Omega(K; R_1, \ldots, R_n) = \{ U \le_c G \colon U \cap K = \ES \land  \forall i\le n \,  U \cap R_i  \neq \ES\},\end{equation}
where $K \sub G$ is compact, and the $R_i \sub G$ are open. 

\begin{proposition}\label{pr:prpr} The map $\Gamma \colon \mathcal {CSI}(\mathcal W) \to \mathcal S(G)$, given by \bc $\mathcal J \mapsto S_\mathcal J = G - \bigcup \mathcal J$,  \ec is a continuous bijection of compact spaces, and hence a homeomorphism. \end{proposition}

\begin{proof} 
Fix a closed-subgroup ideal $\mathcal J$. To show that  $\Gamma$ is continuous at   $\mathcal J$, suppose that $S=\Gamma(\mathcal J) \in \Omega(K; R_1, \ldots, R_n) $,   a basic open set as defined in (\ref{eqn: BO}).  Since the compact open cosets form a basis of the topology of $G$, we   may assume that  each $R_i$ is a compact open coset.   Since $K \sub G-S$ which is open, there are compact open cosets $B_1, \ldots, B_m$ such that $K \sub \bigcup_k B_k \sub G-S$.   Let $\mathfrak L$ be  the basic open     set of           $\mathcal {CSI}(\mathcal W)    $ consisting of the subgroup ideals $\mathcal A$ such that $B_k \in \mathcal A$ for each $k\le m$, and $R_i \not \in \mathcal A$ for each $i \le n$. Clearly $\mathcal J \in \mathcal A$,  and                        $\mathcal H \in \mathfrak L$ implies                 $\Gamma(H) \in  \Omega(K; R_1, \ldots, R_n) $.
\end{proof}
The foregoing proposition yields another proof of the known fact that $\mathcal S(G)$ is    totally disconnected for a t.d.l.c.\ $G$~\cite{Cornulier:11}. In particular,  unless  $\mathcal S(G)$ has  isolated points, it  will be  homeomorphic to  Cantor space. 

\subsection{Concluding  the proof of Theorem~\ref{thm:gr}}  Let $G$ a computably locally compact t.d.l.c.~group.
Recall that Theorem~\ref{thm:gr}  states that $H \leqq_c G$ is computably closed if, and only if, 
  $H$ corresponds to a computable closed-subgroup ideal  in the (computable) dual meet groupoid $\mathcal{W}(G)$ of $G$.
  (This correspondence is described in the previous subsection.) We first prove the following lemma.
  

\begin{lemma} Suppose that a t.d.l.c.~$G$ is given as a computable Baire~presentation based on a tree $T$,
and let $\mathcal W$ denote the corresponding Haar computable copy of  $\mathcal W(G)$.  For an ideal $\mathcal J \sub \mathcal W$, 

\bc $\mathcal J $ is computable $\LR$ the open set $R_{\mathcal J}\sub [T]$ is computable. \ec 
\end{lemma}
\begin{proof} As we already noted after Theorem~\ref{thm:ma}, elements of the groupoid are in a $1$-$1$ effective correspondence with the respective cosets in the group.
We elaborate how exactly this leads to the claimed effective  correspondence between $\mathcal J $ and $R_{\mathcal J}\sub [T]$, and give references to the technical claims in \cite{tdlc} that we use.

\rapf Given $\sigma \in T$, we may assume that $length(\sigma)  > 0$, and thus    $[\sigma]_T$ is compact. Hence by~\cite[Lemma~2.6]{tdlc} one can compute $A_1, \ldots, A_n\in \mathcal W $ such that  $\bigcup_i A_i = [\sigma]$. Then $[\sigma] \sub R$ iff $A_i \in \mathcal J$ for each   $i$.

\lapf By the  definition of  $\mathcal W$   from the computable Baire presentation based on the tree $T$, each $A  \in \mathcal W$ is given in the form $\mathcal K_u= \bigcup_{\eta\in u} [\eta]_T$ (as in~\cite[Def.~2.5]{tdlc}) where $u$ encodes a finite set of nonempty strings on $T$. Then $A \in \mathcal J$ iff $[\eta]\sub R$ for each $\eta$ in $u$. 
\end{proof}

\begin{corollary}
 In the notation of the above, it follows that an open-subgroup ideal [closed-subgroup ideal]   $\mathcal J$ is computable    iff   $R_\mathcal J$ [resp., $S_\mathcal J$] is computable.
 \end{corollary}


By Theorem~\ref{thm:mn} and the discussion after it, given a computably locally compact presentation $C$ of a t.d.l.c.~$G$ we can produce
a computable Baire presentation of the group via a tree $T$ and a computable group-homeomorphism $f: C \rightarrow [T]$ with computable inverse.
By Lemma~\ref{lem:points}, computable closed $S_j$ correspond to computable points in the effective closed presentation of $\mathcal S(G)$   based on $[T]$ or, equivalently, based on $C$.
 Let $\mathcal{W}$ be the Haar computable presentation of the dual meet groupoid given by Theorem~\ref{thm:ma}. 
 Lemma~\ref{lemma:Sj} illustrates that closed-subgroup ideals $\mathcal{J}$ are in 1-1 correspondence with closed subgroups $S_\mathcal{J}$ of $G$.
 By the corollary above, the \emph{computable} closed-subgroup ideals are in a $1$-$1$ correspondence with 
computable closed subgroups $S_\mathcal{J}$ in the presentation given by $[T]$.  This is also clearly uniform, and so is the proof of Lemma~\ref{lem:points}.
This gives  Theorem~\ref{thm:gr}.

\

There is nothing special about computability of points and ideals in the proof above. Indeed, we can computably uniformly transform (names of) $X$-computable points into $X$-computable ideals, and vice versa. Thus, 
as we already pointed out in the introduction,  the proof above really gives a `computable homeomorphism' between $\mathcal {CSI}(\mathcal W)$ and $\mathcal S(G)$.  However, formalising this statement would involve various notions that we did not define in the present paper. 
This is because neither of these objects is a computable Polish space.
Thus, we leave the exact formulation of the more general fact (and the formal verification of it)  as a strong conjecture.

\section{Complexity of index sets}
 Recall from Def.\ \ref{def:loc}  
 that  an \emph{$X$-computably compact structure} on a computable Polish space is
 an $X$-computable functional representing the procedure $N^x\rightarrow  K \supseteq B \ni x$ that works for any point $x$ of the space.
 Recall that  Theorem~\ref{thm:3} states that
$$\{i : G_i \mbox{ is  a properly metrized abelian group and }    \mathbb{R} \leq_c G_i \}$$
is arithmetical.

\subsection{Proof of Theorem~\ref{thm:3}}

\begin{proposition}\label{prop:indexstuff}
The index set of properly metrized  Polish groups is arithmetical ($\Pi^0_3$). 
\end{proposition}
\begin{proof}
Fix $G=G_i$ given by a computable Polish space and two operators acting on the space.
Since we have that $$\overline{B(x, r)}  = cl  \{y: d(x, y) < r\} \subseteq B^c(x,r) =  \{y: d(x, y) \leq r\} \subseteq \overline{B(x, r')}$$
whenever $r'>r$, $G$ is properly metrized iff $\overline{B(x, r)}$ is compact, for all $x,r$, where $x$ can be restricted to range over special points.
First,  view $G$ as a computable Polish space ignoring the group operations; the computable index of the space can be computably reconstructed from $i$.
We claim that it is arithmetical to tell that for every special $x$ and each positive $r \in \mathbb{Q},$ the closure
$\overline{B(x, r)} $ of $B(x, r)$
is compact. The (c.e.) collection of special points in $ B(x, r) = \{y: d(x, y) < r\}$ makes $\overline{B(x, r)}$ a computably Polish space, uniformly in $x, r$.
It is $\Pi^0_3$ to tell whether a given computable Polish space is compact; see \cite{CompComp}.

\begin{fact}\label{fu:fu} Every properly metrized Polish space  admits a $0'$-computably locally compact structure, and this is uniform.
\end{fact}

\begin{proof}
Since every compact Polish space is $0'$-computably compact (and this is uniform), the basic balls and their closures induce a $0'$-computably locally compact ($0'$-c.l.c.) structure on the space.
\end{proof}

Before we proceed  with the proof of  Theorem~\ref{thm:3}, we establish that the index set in the proposition above is sharp. Completely accidentally, it also establishes the $\Pi^0_3$-hardness of the property  of groups from the statement of Theorem~\ref{thm:3}.

  To establish $\Pi^0_3$-completeness in Proposition~\ref{prop:indexstuff}, take the effective sequence $(C_i)$ totally disconnected abelian groups from \cite[Proposition 4.1]{pontr}. 
In this sequence, $C_i$ is profinite iff the $\Pi^0_3$-predicate holds for $i$. If $C_i$ is not profinite, then it is a pro-countable group which is not even locally compact.  In the former case the group is compact, and thus properly metrized.
In the latter case it is not even locally compact, and thus cannot be properly metrized. 
To establish $\Pi^0_3$-hardness in Theorem~\ref{thm:3}, simply  take $C_i \times \mathbb{R}$.

\

We now need to establish that the upper bound in Theorem~\ref{thm:3} is indeed $\Pi^0_3$.
Assume  $G$ is properly metrized (as a space). Using Fact~\ref{fu:fu}, fix $0'$-c.l.c.~structure on the domain of the group.
We claim that it is arithmetical to tell whether the functionals representing the (potential) group operations on a $0'$-c.l.c.~group are well-defined (total). 
This is because they are total if, and only if, for every $n$ their restrictions to $K_n$ are total, where the $K_n$ are the compact sets from  Lemma~\ref{lemma:sigma} (relativized to $0'$) making up the $0'$-c.l.c.~$G$. The totality on each $K_n$ is uniformly $\Pi^0_3$; this is essentially the aforementioned \cite[Lemma 4.2]{pontr}.
Thus, we conclude that totality of the (potential) group operation is an arithmetical property, uniformly in the index of~$G$.

Assuming that the operations are witnessed by total functionals,  
the collection of all points that satisfy the group axioms form an effectively closed set.
Thus, if there are some points that fail the axioms, then there are special points that fail the axioms.  It follows that the group axioms can be tested on special points, which makes the test (uniformly) $\Pi^0_1$ in the index of the group and 
the names of the group operations.
\end{proof}

\begin{remark}
In view of the effective proper metrization result Proposition~\ref{cor:HB},
the proposition above  allows to apply the index set approach to the class of computable compact  groups in a meaningful way.
Also, it contrasts greatly with the  $\Pi^1_1$-completeness of being locally compact Polish space proved by Nies and Solecki~\cite{NSlocal}.
\end{remark}

We return to the proof of Theorem~\ref{thm:3}. By Proposition~\ref{prop:indexstuff}, being a properly metrized group is arithmetical ($\Pi^0_3$). 
Of course, being abelian is a $\Pi^0_1$-property which can be tested on special points.
If $G_i$ is compact, obviously it cannot contain $\mathbb{R}$ as a closed subgroup.
  Recall that we can  arithmetically ($\Pi^0_3$) check whether $G_i$ is compact; however, if we refer to this result, it will increase the complexity of the upper bound on the index set.
  Instead, we assume that $G_i$ has passed all the tests above and, in particular,  is properly metrized. Then it is compact iff 
  $$\exists x \, \mbox{special } \,  \exists r \in \mathbb{Q}  \, \, \, \forall y  \mbox{ special } \, d(x, y) \leq r,$$
  which is $\Sigma^0_2$.
  Thus, we can assume $G_i$ is not compact ($\Pi^0_2$). By 
Fact~\ref{fu:fu}, $G = G_i$ admits a $0'$-computable locally compact structure whose index can be uniformly recovered form $i$ using $0'$.

Using Theorem~\ref{thm:1po} relativized to $0'$, uniformly pass to its  $0'$-computably compact $1$-point compactification $G_i^*$, and uniformly effectively define the  $0'$-computably compact hyperspace $\mathcal K(G^*)$.
Using Theorem~\ref{main:thm}, uniformly produce the $\Pi^0_1(0')$-presentation of $\mathcal S(G)$ inside $\mathcal K(G^*)$.
By Lemma~\ref{0jj}, $\mathcal S(G)$ has a $0''$-computably compact presentation. 

We need the following:

\begin{lemma}[E.g., \cite{pontr}] It is arithmetical ($\Pi^0_1$) whether a given computably compact set is connected. 
\end{lemma} 
\begin{proof}
By \cite[Lemma 4.21]{EffedSurvey}, we can effectively list all clopen non-trivial splits of the space. 
The space is connected iff such a split is never found; this is $\Pi^0_1$.
\end{proof}
Relativising the lemma above to $0''$, we conclude that the following is an arithmetical ($\Pi^0_3$) property:
$$\mathcal S(G) \mbox{ is connected.}$$

In summary, it is arithmetical ($\Pi^0_3$) to tell whether $G_i$ is an abelian properly metrized~group so that its Chabauty space $\mathcal S(G)$  is connected. 
To finish the proof of the theorem,  apply the following result is due to Protasov and Tsybenko~\cite{UKRcon}: 
\emph{Suppose $G$ is a locally compact group.  If its Chabauty space $\mathcal S(G)$ is connected, then some subgroup of $G$   is topologically isomorphic to $(\mathbb{R},+)$. For abelian groups, the converse holds.}

\appendix

\section{Proof of Proposition~\ref{cor:HB}}\label{A}

We use Lemma~\ref{lemma:sigma} to implement an idea similar to that suggested in~\cite{regular_stuff}, which in turn the authors attribute to H. E.~Vaughan.
In the notation of Lemma~\ref{lemma:sigma}, define $$f_n: M \rightarrow [0,1]$$
to be
$$
f_n(x) = \begin{cases}
   d(x, K_n)/ c_n & \text{ if }  d(x, K_n) \leq c_n,   \\
  
  1 & \text{ if } d(x, K_n) \geq c_n,  \\
 
\end{cases}
$$
where the two cases are not mutually exclusive but evidently agree at when $d(x, K_n) = c_n$.
Since $K_n$ is computably compact uniformly in $n$, $(f_n)_{n \in \omega}$ is a uniformly computable  sequence of functions.
Note also that, for each fixed $x \in M$, the sequence $(f_n(x))_{n \in \omega}$ is eventually zero.
Define 
$$f(x) = \sum_{n \in \omega} f_n(x),$$
which is well-defined and is computable.
Finally, define
$$\delta(x,y) = d(x,y) + | f(x) - f(y) |.$$
It is clear that $\delta$ is indeed a metric; we claim it is also equivalent to $d$, i.e.~that the metrics share the same converging sequences.
Since $\delta \geq d$, we need only to show that every $d$-converging sequence also converges with respect to $\delta$; but this follows from
computability (thus, continuity) of $f$.  In particular, we may use the same dense sequence $(x_i)$ in $M$ as we used for $d$, and we get that
$\overline{((x_i), \delta) }$ is a computable Polish space such that $\overline{((x_i), \delta) } = \overline{((x_i), d)} = M$.

We show $\delta$ is proper.
If a closed $\delta$-ball $B^c(x, r)$ is not fully contained in one of the compact sets $K_n$, then there will be an infinite sequence $(y_m)$ so that 
$y_m \notin K_{m+1}$, and thus
$f(y_m) = \sum_n f_n(y_i) \geq m$, contradicting that we must have   $$\delta(x, y_m) = d(x, y_m) + |f(x) + f(y_m)| \leq r.$$ 
it follows that each basic closed ball has to be compact.

To see why $\delta$ and $d$ are effectively compatible, first
recall that $\overline{((x_i), \delta) } = \overline{((x_i), d)} = M$.
The pre-image of each basic open $\delta$-ball $B(x,r)$ under the identity map is simply
$$\{y: d(x,y) + |f(x) -f(y)| <r  \} $$
which is clearly an effectively open set with respect to $d$.
On the other hand, if $(x_i)$ is a fast Cauchy sequence with respect to $\delta$,
$$\delta(x_i, x_{i+1})  =  d(x_i, x_{i+1}) + | f(x_i) -f(x_{i+1}) | < 2^{-i},$$
then clearly $(x_i)$ is also fast Cauchy for $d$. It remains to interpret the latter as the image of $ x = \lim_i x_i$ under $id^{-1}$, where the limit can be taken with respect to either metric.

\section{Proof of Fact~\ref{fact:ChAl}}\label{B}


Recall that the topology of the 1-point compactification 
$G^* = G \cup \{\infty\}$ of $G$ is given by 

-  the open sets of $G$, and

 - the sets $(G \setminus K) \cup \{\infty \}$, where $K$ ranges over compact subsets of $G$.

The Vietoris topology on the set of   closed   subsets of $G^*$ is generated by the sub-basic sets 
$$ \{ F : F\cap V \neq \emptyset\}  \mbox{ and } \{F: F \subseteq U \}  $$
where $U, V$ are open in $G^*$. It is well-known that the  topology induced by the Hausdorff metric and  the Vietoris topology are equivalent.
The Chabauty topology on the closed subgroups of $G$ is generated by the sub-basic sets
 $$\{F : F \cap K = \emptyset \} \mbox{ and } \{F:  F \cap V \neq \emptyset \}$$
 where $K$ is compact and $V$ open in $G$.

\smallskip

 Given a closed subgroup (more generally, subset) $C \subseteq G$, define
 $$i(C)=  C ^* = C \cup \{\infty\}. $$
 This yields as a map from the Chabauty space  of $G$ to the hyperspace $\mathcal K(G^*)$ of the $1$-point compactification of $G$.
 The map is evidently injective.
 We show that it is continuous, and thus a homeomorphism because $\mathcal S(G)$ is    compact.



Open sets in the Vietoris topology on $\mathcal K(G^*)$ are of two kinds, given above. For the first kind, let  $V$ is open in $G \cup \{\infty\}$. If  $\infty \not \in V$, then it is open in $G$ and
  $$i^{-1}\{F^*:  F^* \cap V \neq \emptyset \} = \{F:  F \cap V \neq \emptyset \} $$
  which is open in Chabauty topology.
If $\infty \in V$, the pre-image is the whole space $S(G)$.

Now consider an open set of the second kind, specifically  $\{F^*: F^* \subseteq U \}$ where $U$ is open in $G^*$.
If $U$ does not contain $\infty$, then the set is empty and thus so is its $i$-preimage.
Otherwise,  $U$  has to be of the form 
$(G \setminus K) \cup \{\infty \}$, where $K$ is compact in $G$;  in particular, $\infty \notin K$.

We arrive at $$i^{-1}\{F^*: F^* \subseteq U \} =  \{F : F \subseteq  (G \setminus K) \}  = \{F : F \cap K = \emptyset\}$$
which is open in the Chabauty topology.




%
%

\begin{thebibliography}{HTMM15}

\bibitem[AK00]{AshKn}
C.~Ash and J.~Knight.
\newblock {\em Computable structures and the hyperarithmetical hierarchy},
  volume 144 of {\em Studies in Logic and the Foundations of Mathematics}.
\newblock North-Holland Publishing Co., Amsterdam, 2000.

\bibitem[BdBP12]{BBP}
Vasco Brattka, Matthew de~Brecht, and Arno Pauly.
\newblock Closed choice and a uniform low basis theorem.
\newblock {\em Ann. Pure Appl. Logic}, 163(8):986--1008, 2012.

\bibitem[BHTM21]{bastone}
Nikolay Bazhenov, Matthew Harrison-Trainor, and Alexander Melnikov.
\newblock Computable {S}tone spaces, 2021.

\bibitem[BS14]{norm2}
Ver\'{o}nica Becher and Theodore~A. Slaman.
\newblock On the normality of numbers to different bases.
\newblock {\em J. Lond. Math. Soc. (2)}, 90(2):472--494, 2014.

\bibitem[CCC20]{tdlc4}
I.~Castellano and G.~Corob~Cook.
\newblock Finiteness properties of totally disconnected locally compact groups.
\newblock {\em J. Algebra}, 543:54--97, 2020.

\bibitem[Cor11]{Cornulier:11}
Yves Cornulier.
\newblock On the {C}habauty space of locally compact abelian groups.
\newblock {\em Algebr. Geom. Topol.}, 11(4):2007--2035, 2011.

\bibitem[CR99]{CRindex}
Douglas Cenzer and Jeffrey~B. Remmel.
\newblock Index sets in computable analysis.
\newblock {\em Theoret. Comput. Sci.}, 219(1-2):111--150, 1999.
\newblock Computability and complexity in analysis (Castle Dagstuhl, 1997).

\bibitem[dlH08]{delaharpe}
Pierre de~la Harpe.
\newblock Spaces of closed subgroups of locally compact groups, 2008.

\bibitem[DM20]{MDsurvey}
Rodney~G. Downey and Alexander~G. Melnikov.
\newblock Computable analysis and classification problems.
\newblock In Marcella Anselmo, Gianluca~Della Vedova, Florin Manea, and Arno
  Pauly, editors, {\em Beyond the Horizon of Computability - 16th Conference on
  Computability in Europe, CiE 2020, Fisciano, Italy, June 29 - July 3, 2020,
  Proceedings}, volume 12098 of {\em Lecture Notes in Computer Science}, pages
  100--111. Springer, 2020.

\bibitem[DM23]{EffedSurvey}
R.~Downey and A.~Melnikov.
\newblock Computably compact spaces.
\newblock {\em The {B}ulletin of {S}ymbolic {L}ogic}, 2023.
\newblock To appear.

\bibitem[Dob81]{dobr:81}
V.~Dobrica.
\newblock Constructivizable abelian groups.
\newblock {\em Sibirsk. Mat. Zh.}, 22(3):208--213, 239, 1981.

\bibitem[EG00]{ErGon}
Y.~Ershov and S.~Goncharov.
\newblock {\em Constructive models}.
\newblock Siberian School of Algebra and Logic. Consultants Bureau, New York,
  2000.

\bibitem[Fel62]{Fell}
J.~M.~G. Fell.
\newblock A {H}ausdorff topology for the closed subsets of a locally compact
  non-{H}ausdorff space.
\newblock {\em Proc. Amer. Math. Soc.}, 13:472--476, 1962.

\bibitem[GK02]{GonKni}
S.~Goncharov and J.~Knight.
\newblock Computable structure and antistructure theorems.
\newblock {\em Algebra Logika}, 41(6):639--681, 757, 2002.

\bibitem[GLS03]{GLSOrdered}
S.~Goncharov, S.~Lempp, and R.~Solomon.
\newblock The computable dimension of ordered abelian groups.
\newblock {\em Adv. Math.}, 175(1):102--143, 2003.

\bibitem[GMNT17]{sinf}
Noam Greenberg, Alexander Melnikov, André Nies, and Daniel Turetsky.
\newblock Effectively closed subgroups of the infinite symmetric group.
\newblock {\em Proceedings of the American Mathematical Society}, 146:1, 12
  2017.

\bibitem[GR17]{tdlc3}
Helge Gl\"{o}ckner and C.~R.~E. Raja.
\newblock Expansive automorphisms of totally disconnected, locally compact
  groups.
\newblock {\em J. Group Theory}, 20(3):589--619, 2017.

\bibitem[GW07]{metr}
Tanja Grubba and Klaus Weihrauch.
\newblock On computable metrization.
\newblock {\em Electron. Notes Theor. Comput. Sci.}, 167:345--364, 2007.

\bibitem[HKS20]{topsel}
M.~Hoyrup, T.~Kihara, and V.~Selivanov.
\newblock Degree spectra of homeomorphism types of {P}olish spaces.
\newblock {\em Preprint.}, 2020.

\bibitem[HTM21]{surfaces}
M.~Harrison-Trainor and A.~Melnikov.
\newblock An arithmetic analysis of closed surfaces.
\newblock {\em Preprint}, 2021.

\bibitem[HTMM15]{HMM}
Matthew Harrison-Trainor, Alexander Melnikov, and Antonio Montalb{\'a}n.
\newblock Independence in computable algebra.
\newblock {\em J. Algebra}, 443:441--468, 2015.

\bibitem[HTMN20]{uptohom}
Matthew Harrison-Trainor, Alexander Melnikov, and Keng~Meng Ng.
\newblock Computability of {P}olish spaces up to homeomorphism.
\newblock {\em The Journal of Symbolic Logic}, pages 1--25, 2020.

\bibitem[HW15]{tdlc5}
Karl~H. Hofmann and George~A. Willis.
\newblock Continuity characterizing totally disconnected locally compact
  groups.
\newblock {\em J. Lie Theory}, 25(1):1--7, 2015.

\bibitem[IK21]{IlKi}
Zvonko Iljazovi\'{c} and Takayuki Kihara.
\newblock Computability of subsets of metric spaces.
\newblock In {\em Handbook of computability and complexity in analysis}, Theory
  Appl. Comput., pages 29--69. Springer, Cham, [2021] \copyright 2021.

\bibitem[Kam01]{Kamo}
Hiroyasu Kamo.
\newblock Effective contraction theorem and its application.
\newblock In {\em Computability and complexity in analysis ({S}wansea, 2000)},
  volume 2064 of {\em Lecture Notes in Comput. Sci.}, pages 88--100. Springer,
  Berlin, 2001.

\bibitem[Khi98]{Khi}
N.~Khisamiev.
\newblock Constructive abelian groups.
\newblock In {\em Handbook of recursive mathematics, {V}ol.\ 2}, volume 139 of
  {\em Stud. Logic Found. Math.}, pages 1177--1231. North-Holland, Amsterdam,
  1998.

\bibitem[KMK23]{CoToGr}
Koh, A.~Melnikov, and K.M.Ng.
\newblock Computable topological groups.
\newblock 2023.
\newblock Journal of Symbolic Logic, to appear.

\bibitem[LMN21]{lupini}
M.~Lupini, A.~Melnikov, and A.~Nies.
\newblock Computable topological abelian groups.
\newblock {\em Preprint.}, 2021.

\bibitem[LR78]{LaRothesis}
Peter~Edwin La~Roche.
\newblock {\em CONTRIBUTIONS TO RECURSIVE ALGEBRA}.
\newblock ProQuest LLC, Ann Arbor, MI, 1978.
\newblock Thesis (Ph.D.)--Cornell University.

\bibitem[LR81]{LaRo1}
Peter La~Roche.
\newblock Effective {G}alois theory.
\newblock {\em J. Symbolic Logic}, 46(2):385--392, 1981.

\bibitem[Mal62]{Mal}
A.~Mal{\cprime}cev.
\newblock On recursive {A}belian groups.
\newblock {\em Dokl. Akad. Nauk SSSR}, 146:1009--1012, 1962.

\bibitem[Man89]{point-paper}
Mark Mandelkern.
\newblock Metrization of the one-point compactification.
\newblock {\em Proc. Amer. Math. Soc.}, 107(4):1111--1115, 1989.

\bibitem[Mel10]{mel1}
A.~Melnikov.
\newblock Computable ordered abelian groups and fields.
\newblock In {\em Programs, proofs, processes}, volume 6158 of {\em Lecture
  Notes in Comput. Sci.}, pages 321--330. Springer, Berlin, 2010.

\bibitem[Mel14]{melbsl}
Alexander~G. Melnikov.
\newblock Computable abelian groups.
\newblock {\em The Bulletin of Symbolic Logic}, 20(3):315--356, 2014.

\bibitem[Mel18]{pontr}
Alexander Melnikov.
\newblock Computable topological groups and {P}ontryagin duality.
\newblock {\em Trans. Amer. Math. Soc.}, 370(12):8709--8737, 2018.

\bibitem[MM18]{MeMo}
Alexander Melnikov and Antonio Montalb\'{a}n.
\newblock Computable {P}olish group actions.
\newblock {\em J. Symb. Log.}, 83(2):443--460, 2018.

\bibitem[MN79]{MetNer79}
G.~Metakides and A.~Nerode.
\newblock Effective content of field theory.
\newblock {\em Ann. Math. Logic}, 17(3):289--320, 1979.

\bibitem[MN13]{CompComp}
Alexander~G. Melnikov and Andr{\'e} Nies.
\newblock The classification problem for compact computable metric spaces.
\newblock In {\em The nature of computation}, volume 7921 of {\em Lecture Notes
  in Comput. Sci.}, pages 320--328. Springer, Heidelberg, 2013.

\bibitem[MN22]{tdlc}
A.~Melnikov and A.~Nies.
\newblock Computably locally compact totally disconnected groups.
\newblock Submitted, available at arXiv.org as 2204.09878., 2022.

\bibitem[MN23]{separ}
Alexander~G Melnikov and Keng~Meng Ng.
\newblock Separating notions in effective topology.
\newblock {\em International Journal of Algebra and Computation},
  33(08):1687--1711, 2023.

\bibitem[NS15]{NSlocal}
Andr{\'{e}} Nies and Slawomir Solecki.
\newblock Local compactness for computable polish metric spaces is {$\Pi_1^1$}
  -complete.
\newblock In {\em Evolving Computability - 11th Conference on Computability in
  Europe, CiE 2015, Bucharest, Romania, June 29 - July 3, 2015. Proceedings},
  pages 286--290, 2015.

\bibitem[Nur74]{Nurt}
A.~Nurtazin.
\newblock {\em Computable classes and algebraic criteria of autostability}.
\newblock Summary of Scientific Schools, Math. Inst. SB USSRAS, Novosibirsk,
  1974.

\bibitem[Pau16]{ArnoSurvey}
Arno Pauly.
\newblock On the topological aspects of the theory of represented spaces.
\newblock {\em Computability}, 5(2):159--180, 2016.

\bibitem[Pau19]{lc1}
Arno Pauly.
\newblock Effective local compactness and the hyperspace of located sets.
\newblock {\em CoRR}, abs/1903.05490, 2019.

\bibitem[PSZ20]{ArnoHaar}
Arno Pauly, Dongseong Seon, and Martin Ziegler.
\newblock {Computing Haar Measures}.
\newblock In Maribel Fern{\'a}ndez and Anca Muscholl, editors, {\em 28th EACSL
  Annual Conference on Computer Science Logic (CSL 2020)}, volume 152 of {\em
  Leibniz International Proceedings in Informatics (LIPIcs)}, pages
  34:1--34:17, Dagstuhl, Germany, 2020. Schloss Dagstuhl--Leibniz-Zentrum fuer
  Informatik.

\bibitem[PT83]{UKRcon}
\={I}.~V. Protasov and Yu. V.~and Tsybenko.
\newblock Connectedness in the space of subgroups.
\newblock {\em Ukrain. Mat. Zh.}, 35(3):382--385, 1983.

\bibitem[Rab60]{abR}
M.~Rabin.
\newblock Computable algebra, general theory and theory of computable fields.
\newblock {\em Trans. Amer. Math. Soc.}, 95:341--360, 1960.

\bibitem[Smi79]{SmithThesis}
Rick~L. Smith.
\newblock {\em The theory of profinite groups with effective presentations}.
\newblock ProQuest LLC, Ann Arbor, MI, 1979.
\newblock Thesis (Ph.D.)--The Pennsylvania State University.

\bibitem[Smi81]{Smith1}
Rick~L. Smith.
\newblock Effective aspects of profinite groups.
\newblock {\em J. Symbolic Logic}, 46(4):851--863, 1981.

\bibitem[Soa87]{Soa}
R.~Soare.
\newblock {\em Recursively enumerable sets and degrees}.
\newblock Perspectives in Mathematical Logic. Springer-Verlag, Berlin, 1987.
\newblock A study of computable functions and computably generated sets.

\bibitem[Wei03]{WeirauchSpacesStuff}
Klaus Weihrauch.
\newblock Computational complexity on computable metric spaces.
\newblock {\em MLQ Math. Log. Q.}, 49(1):3--21, 2003.

\bibitem[Wes15]{tdlc2}
Phillip Wesolek.
\newblock Elementary totally disconnected locally compact groups.
\newblock {\em Proc. Lond. Math. Soc. (3)}, 110(6):1387--1434, 2015.

\bibitem[Wil15]{tdlc1}
George~A. Willis.
\newblock The scale and tidy subgroups for endomorphisms of totally
  disconnected locally compact groups.
\newblock {\em Math. Ann.}, 361(1-2):403--442, 2015.

\bibitem[WJ87]{regular_stuff}
Robert Williamson and Ludvik Janos.
\newblock Constructing metrics with the {H}eine-{B}orel property.
\newblock {\em Proc. Amer. Math. Soc.}, 100(3):567--573, 1987.

\bibitem[WZ99]{lc3}
K.~Weihrauch and X.~Zheng.
\newblock Effectiveness of the global modulus of continuity on metric spaces.
\newblock {\em Theoretical Computer Science}, 219(1):439 – 450, 1999.

\bibitem[XG09]{lc2}
Y.~Xu and T.~Grubba.
\newblock On computably locally compact {Hausdorff} spaces.
\newblock {\em Mathematical Structures in Computer Science}, 19(1):101 -- 117,
  2009.

\end{thebibliography}
%

\def\cprime{$'$} \def\cprime{$'$} \def\cprime{$'$} \def\cprime{$'$}

\end{document}